\documentclass[a4paper,10pt]{article}
\usepackage{graphicx}
\usepackage{float}
\usepackage{amsthm}
\usepackage{amsfonts}
\usepackage{indentfirst}
\usepackage{amsmath}
\usepackage{amssymb}
\usepackage{color}

\newtheorem{theo}{Theorem}
\newtheorem{lemma}{Lemma}
\newtheorem{cor}{Corollary}
\newtheorem{prop}{Proposition}
\newtheorem{fig}{Figure}

\newtheorem*{ann}{Assumption $\mathcal{H}(\beta)$}

\def\N{\mathbb{N}}

\def\P{\mathbb{P}}

\def\E{\mathbb{E}}

\def\t{\textrm}
\def\d{\textrm{d}}

\def\w{\widetilde}

\def\ind{{\mathchoice {\rm 1\mskip-4mu l} {\rm 1\mskip-4mu l}
{\rm 1\mskip-4.5mu l} {\rm 1\mskip-5mu l}}}

\newcommand{\be} {\begin{equation}}
\newcommand{\ee} {\end{equation}}
\newcommand{\bea} {\begin{eqnarray}}
\newcommand{\eea} {\end{eqnarray}}
\newcommand{\Bea} {\begin{eqnarray*}}
\newcommand{\Eea} {\end{eqnarray*}}

\begin{document}
\title{Upper large deviations for Branching Processes \\ in Random Environment with heavy tails}
\author{Vincent Bansaye \footnote{CMAP, Ecole Polytechnique, Palaiseau} 
, Christian Boeinghoff \footnote{Department of mathematics, Goethe-university Frankfurt/Main} 
}
\maketitle \vspace{1.5cm}

\begin{abstract} Branching Processes in a Random Environment (BPREs) $(Z_n:n\geq0)$  are a generalization of Galton Watson processes where in each generation the reproduction law is picked randomly in an i.i.d. manner. We determine here the upper large deviation of the process when the reproduction law may have heavy tails.  The behavior of BPREs is related to the associated random walk of the environment, whose increments are distributed like the logarithmic mean of the offspring distributions. We obtain an expression of the upper rate function of  $(Z_n:n\geq0)$, that  is the limit  of  $-\log \mathbb{P}(Z_n\geq e^{\theta n})/n$ when $n\rightarrow \infty$. It   depends on the rate function of the associated random walk of the environment, the logarithmic cost of survival $\gamma:=-\lim_{n\rightarrow\infty} \log \mathbb{P}(Z_n>0)/n$ and the polynomial decay $\beta$  of the tail distribution of $Z_1$. We give interpretations of this rate function in terms of  the least costly ways for the process $(Z_n: n \geq 0)$
of attaining extraordinarily large values and describe the phase transitions. We derive then the rate function when the reproduction law does not have heavy tails, which generalizes the  results of  B\"oinghoff and Kersting (2009) and  Bansaye and Berestycki (2008) for upper large deviations. Finally, we specify the upper large deviations for  the Galton Watson processes with heavy tails.
\end{abstract}


\maketitle

\section{Introduction}
Branching processes in a random environment have been introduced in \cite{athreya71} and \cite{smith69}. 
In each generation, an offspring distribution is chosen at random, independently from one generation to the other.  We can think of  a population of plants which have a one year life-cycle. Each year the weather conditions (the environment) vary, which impacts the reproductive success of the plant. Given the climate, all the plants reproduce independently according to the same  mechanism.\\ 
Initially, these processes have mainly been studied under the assumption of i.i.d. offspring distributions which are geometric, or more generally, linear fractional \cite{afa80, kozlov76}. Then, the case of general offspring distributions has attracted attention \cite{kersting052, kersting053, birkner05, kersting00}. 

Recently, several results about  large deviations of branching processes in random environment for offspring distributions with weak tails have been proved. More precisely, \cite{kozlov06} ensures that   $\mathbb{P}(Z_n\geq \exp(\theta n))$  is equivalent to $I(\theta)\P(S_n\geq \theta n)$ for geometric offspring distributions and $\theta$ large enough. In \cite{bansaye08}, the authors give a general upper  bound for the rate function and compute it  when each individual leaves at least one offspring, i.e. $\mathbb P(Z_1=0)=0$. Finally \cite{BK09} gives an expression of  the upper rate function when the reproduction laws have at most geometric tails, which excludes  heavy tails. \\
Exceptional growth of BPREs can be due to an exceptional environment  and/or to exceptional reproduction in some given environment. In this paper, we focus on large deviation probabilities when the offspring distributions  may have  heavy tails and the exceptional reproduction of a single individual can now contribute to the large deviation event.  This leads us to consider new auxiliary power series and higher order derivatives of generating functions for the proof.  \\

$\quad$ Let us give now the formal definition of the process $(Z_n:n\in\N)$, $\N=\{0,1,2,3,\ldots\}$, by considering a random probability generating function $f$ and  a sequence $(f_n: n\geq 1)$
of i.i.d. copies of $f$ which serve as random environment. Conditionally on the environment
 $(f_n:n\geq 1)$, individuals at generation $n$ reproduce
independently of each other and their offsprings have generating function $f_{n+1}$. 
We denote by $Z_n$ 
 the number of particles in  generation $n$
and $Z_{n+1}$ is the sum of $Z_n$ independent random variables with generating function $f_{n+1}$.
That is, for every $n\geq 0$,
$$\E\big[s^{Z_{n+1}}\vert Z_0,\dots,Z_n; \ f_1,\dots,f_{n+1}\big]=f_{n+1}(s)^{Z_n} \ \t{a.s.} \qquad (0\leq s\leq 1) .$$
In the whole paper, we denote by $\P_k$ the probability associated with $k$ initial particles and then,  we have for all $k\in\N$ and $n\in\N$,
$$\E_k[s^{Z_{n}}\ \vert \ \ f_1, ..., \ f_n]=[f_{1}\circ \cdots\circ f_{n}(s)]^k \quad \t{a.s.} \qquad (0\leq s\leq 1).$$
Unless otherwise specified, the initial population size is $1$.\\

We introduce the exponential rate of decay of the survival probability
\begin{eqnarray}
 \gamma &:=& \lim_{n\rightarrow\infty} -\frac{1}{n} \log \P(Z_n>0) \ .\label{lem31}
\end{eqnarray}
The fact that the limit exists and $0\leq \gamma <\infty$ is classical (see \cite{BK09}) since
the sequence $(-\log \P(Z_n>0))_n$ is subadditive and nonnegative (see  \cite{dembo}).  Essentially, $\gamma=0$ in the supercritical or critical case $(\E(X)\leq 0$) and  
$$\gamma=-\log\big(\inf\{ \E(\exp(sX) : s\in [0,1]\}\big)$$
in the subcritical case. In this latter case,  $\gamma=-\log(\E(f'(1)))$ in the strongly or intermediate subcritical case) ($\E(X\exp(X))\leq 0$) whereas
$\gamma> -\log(\E(f'(1)))$ in the weakly subcritical case ($\E(X\exp(X))>0$). We refer to \cite{GKV03} for more precise asymptotic results on the survival probability in the subcritical case.

$\quad$ Many  properties of $Z$ are mainly determined by the \textbf{random walk associated with the environment} 
$$S_0=0, \qquad S_n-S_{n-1}=X_n \quad (n\geq 1).$$ 
where 
\begin{eqnarray}
 X_n&:=& \log (f_n'(1))  \quad (n\geq 1), \nonumber
\end{eqnarray}
are i.i.d. copies of the logarithm of the mean number of offsprings  
\begin{eqnarray}
 X&:=& \log (f'(1))  \nonumber
\end{eqnarray}
If $Z_0=1$, we get for the conditioned means of $Z_n$
\begin{eqnarray}
\E[Z_n| f_1,\ldots,f_n] &=& e^{S_n} \quad \mbox{a.s.}  \label{ew1}
\end{eqnarray}

In the whole paper, we assume that there exists  $s>0$ such that the moment generating function $\E[\exp(sX)]$ 
is finite and we introduce the rate function $\Lambda$ of the random walk $(S_n:n\in\N)$ 
\begin{eqnarray}
 \Lambda(\theta) &:=& \sup_{ \lambda\geq0} \big\{ \lambda \theta -\log(\E[\exp(\lambda X)]) \big\} \label{rate}
\end{eqnarray}
As $\Lambda$ is convex and lower semicontinuous, there is at most one $\theta\geq0$ with $\Lambda(\theta) \neq \Lambda(\theta+)$. In this case, $\Lambda(\theta+)=\infty$ (see e.g \cite{hollander}, \cite{dembo}). Usually, $\Lambda$ is defined as the Legendretransform of $\log(\E[\exp(\lambda X)])$ and the supremum in (\ref{rate}) is taken over all $\lambda\in\mathbb{R}$. Here, we are only interested in upper deviations, thus setting $\Lambda(\theta)=0$ for $\theta\leq \mathbb{E}[X]$ is convenient.\qed \\[0.3cm]

We write $L=L(f)$ for the random variable associated with the probability generating function $f$:
$$\E[s^{L}\ \vert \ f]=f(s) \quad (0\leq s\leq 1) \quad \t{a.s.}$$
 and we denote by $m=m(f)$ its expectation:
$$m:=f'(1)=\E[L|f]<\infty\quad \t{a.s.}$$

\section{Main results and interpretation}
$\quad$ We describe here  the upper large deviations of the branching process $(Z_n:n\in\N)$ when the offspring distributions may have heavy tails. This means that the probability
that one individual gives birth to an exponential number of offsprings may decrease 'only exponentially'.
More precisely, we work with the following assumption, which ensures that
the tail of the offspring distribution of an individual, conditioned to be positive, decays at least
 with exponent $\beta \in (1,\infty)$ (uniformly with respect to the environments).
\begin{ann}
\label{as1}
 There exists a constant  $0<d<\infty$ such that for every $z\geq 0$, 
\bea
\P(L>z \ \vert \ f, L>0)&\leq& d\cdot (m\wedge 1)\cdot z^{-\beta} \quad \t{a.s.} \label{as22} 
\nonumber 	  
\eea
\end{ann}
The rate function $\psi$ we establish  and interpret below depends on  $\gamma$, $\beta$ and $\Lambda$ and is defined by 
\begin{eqnarray}
\psi(\theta):=\inf_{t\in[0,1],  s\in [0,\theta]} \Big\{t\gamma+\beta s+(1-t)\Lambda((\theta-s)/(1-t))\Big\} \ \ (=\psi_{\gamma,\beta,\Lambda}(\theta)). \label{psidef} 
\end{eqnarray}
 Note that in the supercritical case (i.e.  $\E[\log(f'(1))]>0$), $\psi$ simplifies to
$$\psi(\theta)=\inf_{s\in[0,\theta]} \{\beta s+\Lambda(\theta-s)\}.$$  
\begin{theo} \label{theo1}
Assume that for some $\beta\in(1,\infty)$, $\log(\P(Z_1>z))/\log(z) \stackrel{z\rightarrow \infty}{\longrightarrow } -\beta$ and that additionally $\mathcal{H}(\beta)$ holds. Then for every $\theta\geq 0$, 
$$-\frac{1}{n} \log(\P(Z_n\geq e^{\theta n}))\stackrel{n\rightarrow \infty}{\longrightarrow } \psi(\theta).$$
\end{theo}
The assumptions  in this  Theorem ensure  that   the offspring distributions associated to 'some environments' have polynomial tails with exponent $-\beta$, and no tail distribution  exceeds this exponent. \\

The upper bound is proved in section 3, while the proof of the lower bound is given in sections \ref{beta12} and \ref{beta2} by distinguishing the case $\beta \in (1,2]$ and the case $\beta>2$. The proof for $\beta>2$ is technically more involved since it requires higher order derivatives of generating functions and we adapt  in section \ref{beta2}  the arguments of the proof for $\beta\in(1,2]$. \\

\textbf{Remark:} This theorem still  holds if we just assume that there exists a slowly varying function $l$ such that
$$ \P(L>z \ \vert \ f, L>0)\leq d \cdot(m\wedge 1)\cdot \l(z) z^{-\beta} \quad \t{a.s.}$$
instead of assumption $\mathcal{H}(\beta)$. Indeed, by properties of slowly varying functions (see  \cite{BGT}, proposition 1.3.6, page 16), for any $\epsilon>0$, there exists a constant $d_{\epsilon}$ such that $ \P(L>z \ \vert \ f, L>0)\leq d_{\epsilon} \cdot(m\wedge 1)\cdot z^{-\beta+\epsilon} \ \t{a.s}.$
As for fixed  $\theta\geq 0$, $\psi_{\gamma, \beta, \Lambda}$ is continuous in $\beta$, letting  $\epsilon\rightarrow 0$ yields the claim.\qed \\

$\qquad $Let us give two consequences of this result. First, we  derive a large deviation result for offspring distributions without heavy tails by letting $\beta\rightarrow \infty$, which generalizes Theorem 1 in \cite{BK09}.

\begin{cor}
If assumption $\mathcal{H}(\beta)$ is fulfilled for every $\beta>0$, then for every $\theta\geq 0$,
$$-\frac{1}{n} \log(\P(Z_n\geq e^{\theta n}))\stackrel{n\rightarrow \infty}{\longrightarrow } \inf_{t\in[0,1]} \Big\{t\gamma+(1-t)\Lambda(\theta/(1-t))\Big\}.$$
\end{cor}
For example, this result holds if the offspring distributions are bounded ($\P(L\geq a \ \vert \ f)=0$ a.s. for some constant $a$)
or if  $\P(L> z \ \vert \ f, \ L>0)\leq c\exp(-z^{\alpha})$ a.s. for some constants $c,\alpha\geq 0$.\\

$\qquad$ Second we deal with the Galton Watson case,  so the environment is not random and $f$ is deterministic, meaning $\Lambda(\theta)=\infty$ for $\theta>\log m$ and $\Lambda(\log m)=0$. . We refer
to \cite{BB, Rouault} for precise results for large deviations without heavy tails. 
For the decay rate of the survival probability, it is known that (see \cite{AN}) in the subcritical case ($m<1$)
\begin{eqnarray}
 \gamma&=& -\log m \nonumber
\end{eqnarray}
and $\gamma=0$ in the critical ($m=1$) and supercritical ($m>1$) case. Thus, in the subcritical case,
\begin{eqnarray}
 \psi(\theta)&=&-\log m + \beta \theta \ . \nonumber 
\end{eqnarray}
In the critical and supercritical case, it remains to minimize
\begin{eqnarray}
 \psi(\theta)=\inf_{s\in [0,\theta] } \{\beta s + \Lambda(\theta-s)\} \ , \nonumber
\end{eqnarray}
where $\Lambda(\theta)=0$ for $\theta\leq \log m$ and $\Lambda(\theta)=\infty$ for $\theta>\log m$. Hence,
\begin{eqnarray}
 \psi(\theta)=\beta (\theta-\log m) \ . \nonumber
\end{eqnarray}

\paragraph{Path interpretation of the rate function.}
The rate function gives the exponential decay rate of the probability of reaching exceptionally large values, namely
\begin{eqnarray}
\mathbb{P}(Z_n\geq \theta n)&=&\exp(-\psi(\theta)n+o(n)) . \nonumber 
\end{eqnarray}
We consider the following 'natural  paths' which reaches  extraordinarily large values, i.e a path  which realizes $\{Z_n\geq \exp(\theta n)\}$ for $n\gg1$ and $\theta> \E[\log(f'(1))]$. At the beginning, up to time $\lfloor t n\rfloor$, there is a period without growth, that is the process just survives. The probability of this event decreases as  $\exp(-\gamma \lfloor tn\rfloor)$. At time $\lfloor tn\rfloor$, there are very few individuals and one individual has exceptionally many offsprings, namely $\exp(sn)$-many. The probability of this event
 is given by $\P(Z_1\geq \exp(sn))$ so it is of the order of  $\exp(-\beta sn)$. Then the process grows exponentially according to its expectation in a good environment to reach $\exp(\theta n)$.  That is $S$ grows linearly such that $S_n-S_{\lfloor nt\rfloor} \approx  [\theta-s] n$ 
and the probability to observe this exceptionally good environment sequence  decreases
as  $\exp(-(1-t)\Lambda((\theta-s)/(1-t))n)$.  The most probable path to reach  extraordinary large values $\exp(\theta n)$ at time $n$ is then obtained by minimizing 
 the sum of these three 'costs' $\gamma t$, $\beta s$ and $(1-t)\Lambda((\theta-s)/(1-t))$, which gives the rate function $\psi$. \\

$\qquad$The optimal strategy to realize the large deviation event is given 
by the bivariate value  $(t_{\theta},s_{\theta})$ such that
$$ \psi(\theta)=t_{\theta}\gamma+\beta s_{\theta}+(1-t_{\theta})\Lambda((\theta-s_{\theta})/(1-t_{\theta})) \ .$$
More formally, following the proof of \cite{bansaye08}, we should be able to prove the uniqueness of  $(t_{\theta},s_{\theta})$ (except for degenerated situations) and the forthcoming  trajectorial result. But the proof become very heavy and technical. Conditionally on   $Z_n \leq e^{ c n }$, we expect that
$$\sup_{t\in[0,1]} \big\{ \big\vert \log(Z_{[tn]})/n- f_{\theta}(t) \big\vert\big  \} \ \stackrel{n\rightarrow \infty}{\longrightarrow }0$$
in probability in the sense of the uniform norm where
$$f_{\theta}(t):= \left\{\begin{array}{ll}  0, \qquad &  \t{if} \ t\leq t_{\theta} \\
                         \beta s_{\theta}+\frac{c}{1-t_{\theta}}(t-t_{\theta}), \quad & \t{if} \ t> t_{\theta}.
\end{array}
\right.
$$

\begin{fig} Representation of $t\in[0,1]\rightarrow f_{\theta}(t)$.
\label{graphlog}
\begin{center}
$\includegraphics[scale=0.45]{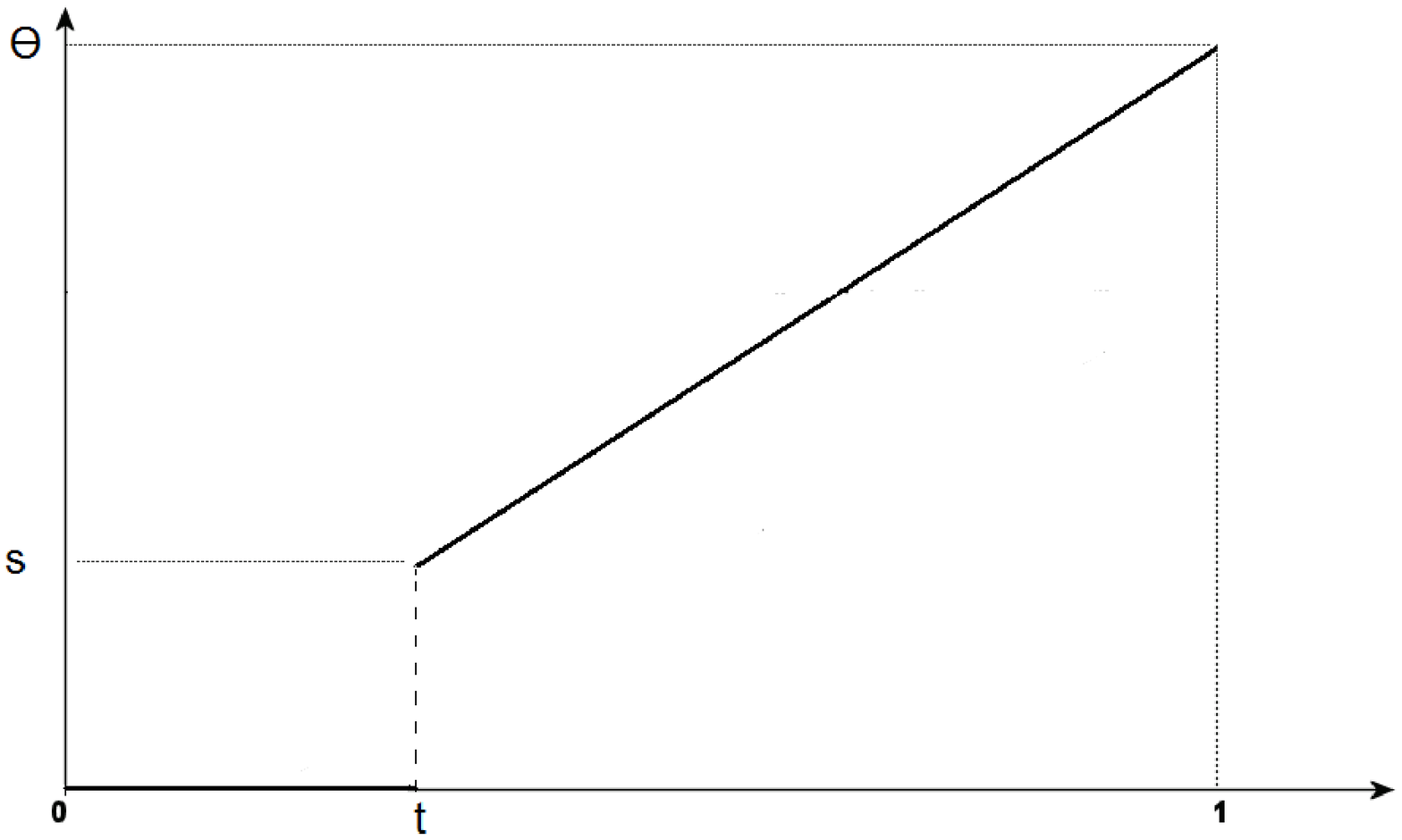}$
\end{center} 
\end{fig}

As detailed in the next paragraph, several  strategies  may occur following the regime of the process  and the value of $\theta$. Except in degenerated cases when the associated path is not unique,   we prove below using convexity arguments that  the  jump occurs  at the beginning  $(s_{\theta}>0\Rightarrow t_{\theta}=0)$
 or at the end $(t_{\theta}=1)$ of the trajectory. Thus upper large deviation events correspond to one of the following trajectories. \\
\begin{fig} Representation of the possible trajectories of the path associated to upper large deviations.
\label{quatre}
\begin{center}
$\includegraphics[scale=0.16]{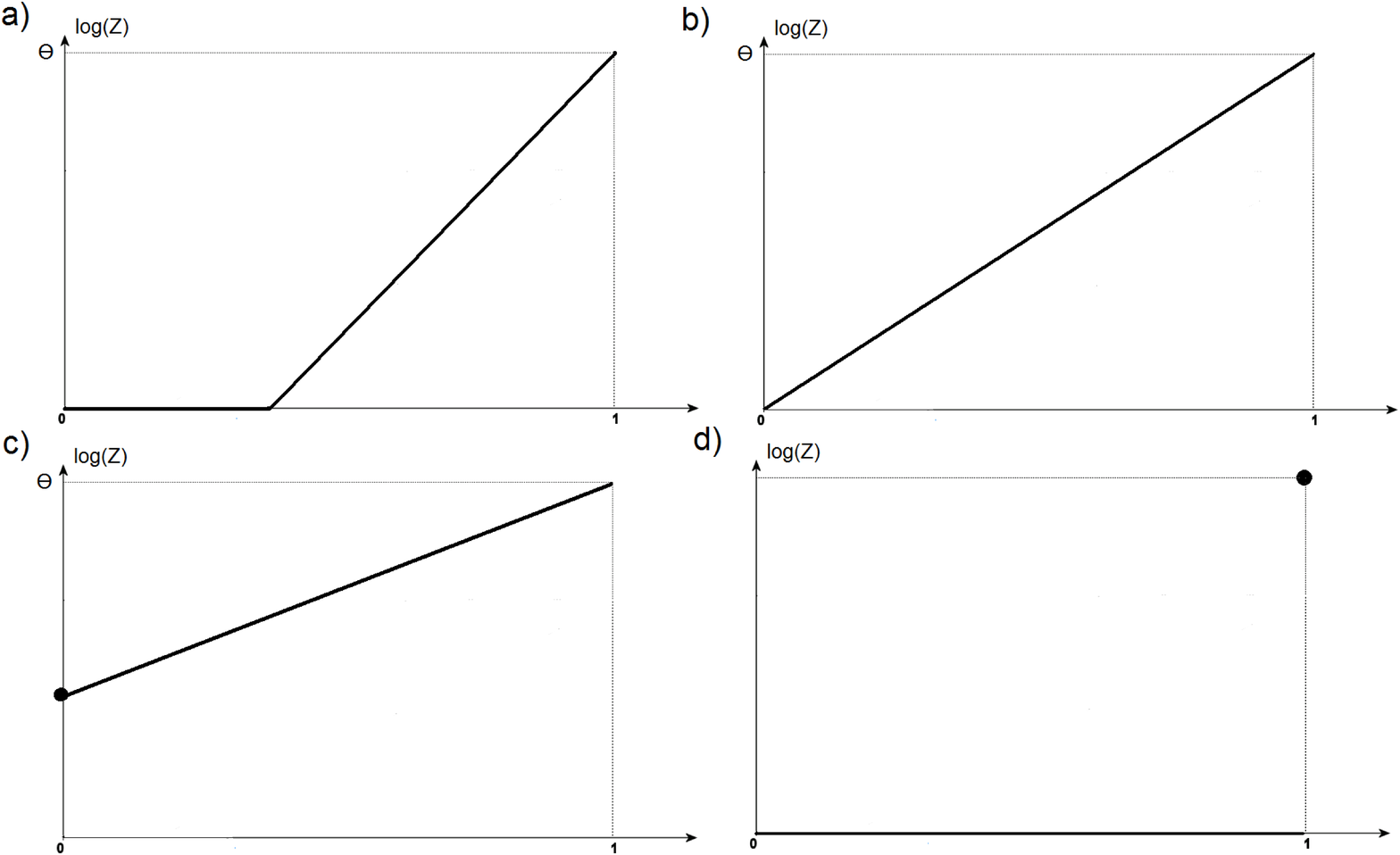}$
\end{center} 
\end{fig}


Obviously, keeping the population size small during a first period $(t_{\theta}>0)$ and growing later (Figure \ref{quatre} a)) 
 can be relevant only in the subcritical case. Actually we see below that the situation from Figure \ref{quatre} a) only occurs if the process is strongly subcritical, as previously observed in \cite{BK09} without heavy tails. In the subcritical case, if $\Lambda'(0)>\beta$,  the associated optimal way is to keep the population size small by just surviving until the final time and then jump to the final value (Figure
\ref{quatre} d)). Then
  the population of size $\exp(\theta n)$ comes from a single parent of one of the last generations. The phase transitions are described later.

In the supercritical case ($\E[\log(f'(1))]>0$), the process does not stay at zero ($t_{\theta}=0$) but may  jump at time $t=0$, and then goes in straight line to reach $\theta$. This corresponds to Figure \ref{quatre} b) and c). 

In the Galton Watson case with mean offspring $m$, the good strategy is either to survive until the final time and jump to the desired value $\theta$ (if $m\leq 1$), or to jump to $\theta-\log(m)$ and then grow normally (if $m>1$).

\paragraph{Graphical construction of the rate function.}
Here, we give another characterization of $\psi$, which will be useful to describe the strategy for upper large deviations
in function of $\theta$. 
As proved in Lemma \ref{exprPsi} (see appendix), $\psi$ is  the largest convex function which satisfies
for all $x,\theta \geq0$
$$
 \psi(0)=\gamma, \quad   \psi(\theta)\leq \Lambda(\theta), \quad \psi(\theta+x)\leq \psi(\theta)+\beta x.$$

The first condition  plays a role iff $\Lambda(0)>\gamma$, which corresponds to the strongly subcritical case (i.e.
$\E[f'(1)\log(f'(1))]<0$, see \cite{GKV03}). Indeed if $\E[X\exp(X)]<0$, then the differentiation of  $s\rightarrow \E[\exp(sX)]$ in $s=1$ is negative and $\Lambda(0)=\sup\{-\log(\E[\exp(sX)] ) : s\geq 0\}>-\log(\E[\exp(X)])=\gamma$. If $\E[X\exp(X)]\geq 0$,   \cite{GKV03} and the definition of $\Lambda$ ensure that  both $\gamma$ and $\Lambda(0)$
are equal to $-\log(\E[\exp(\nu X)])$ where $\nu$ is characterized by $\E[X\exp(\nu X)]=0$. \\

This characterization leads us to construct $\psi$ by three pieces separated by $\theta^*$ and $\theta^{\dagger}$. 
\begin{fig}The following picture gives $\psi$ in the strongly subcritical case:\\
\label{ratefig}
\begin{center}
$\includegraphics[scale=0.33]{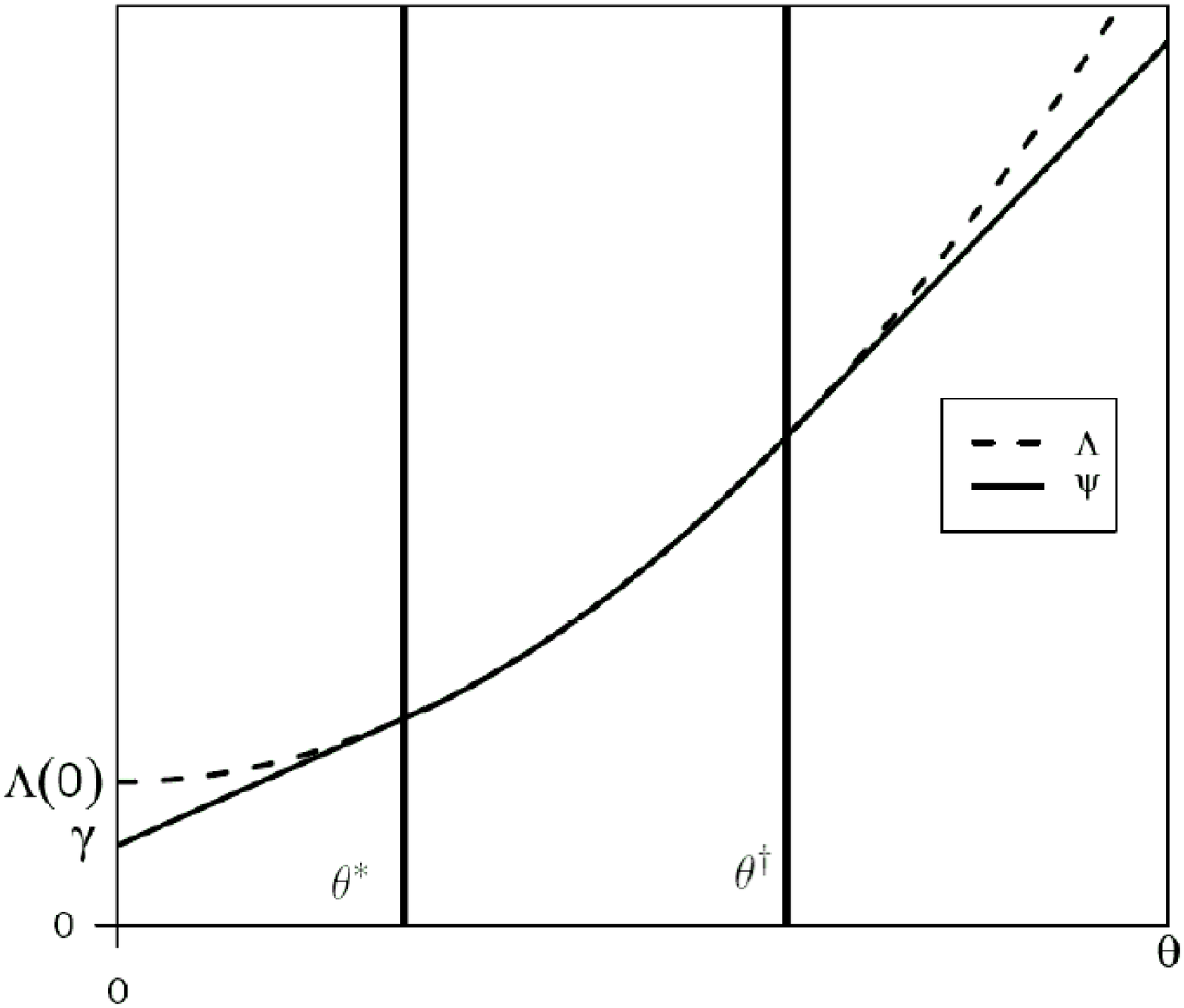}$
\end{center} 
\end{fig}

More explicitly, we define $\chi$ as the largest convex function which satisfies
\begin{eqnarray}
 \chi(0)\leq \gamma, && \chi (\theta)\leq \Lambda(\theta) \nonumber
 \end{eqnarray}
for all $\theta \geq 0$. This function  is the rate function of $Z$ in case of offspring distributions having at most geometric tails (see \cite{BK09}) and is given by
\begin{eqnarray}
 \chi (\theta) &=& \left\{ \begin{array}{l@{\quad,\quad}l}
                              \gamma\left(1-\frac{\theta}{\theta^*}\right)+\frac{\theta}{\theta^*} \Lambda(\theta^*) & \mbox{if} \ \theta<\theta^* \\
\Lambda(\theta) & \mbox{else}
                             \end{array} \right. \nonumber
\end{eqnarray}
where $0\leq \theta^*\leq \infty$ is defined by
\begin{eqnarray}
 \frac{\Lambda(\theta^*)-\gamma}{\theta^*} &=& \inf_{\theta\geq 0} \frac{\Lambda(\theta)-\gamma}{\theta}. \label{thetastar}
\end{eqnarray}
Now define
\begin{eqnarray}
 \theta^{\dagger} &=& \sup\Big\{\theta\geq \max\{0,\E[X]\} : \  \chi'(\theta)\leq \beta \ \mbox{and} \ \chi(\theta)<\infty\Big\}. \label{thetadagger}
\end{eqnarray}
Then 
\begin{eqnarray}
 \psi(\theta) &=& \left\{ \begin{array}{l@{\quad,\quad}l}
                              \chi(\theta) & \mbox{if} \ \theta\leq \theta^{\dagger} \\
\beta\theta-\log (\E[e^{\beta X}]) & \mbox{else}
                             \end{array} \right. \ . \label{rate2}
\end{eqnarray}

\paragraph{Phase Transitions}
Let us first describe the phase transitions (of order two) of the rate function and the  strategies associated with when $\theta^{\dagger}>0$. For that we use the following expression,
\begin{eqnarray}
 \psi(\theta) &=& \left\{ \begin{array}{l@{\quad,\quad}l}
                             \gamma(1-\frac{\theta}{\theta^*}) + \frac{\theta}{\theta^*} \Lambda(\theta^*) & \mbox{if} \ \theta\leq \theta^{*} \\
\Lambda(\theta) & \mbox{if} \ \theta^*<\theta< \theta^{\dagger}\\
 \beta(\theta-\theta^\dagger) + \Lambda(\theta^\dagger) & \mbox{if} \ \theta\geq \theta^{\dagger} 
                             \end{array} \right. \ . \nonumber
\end{eqnarray}
which can be guessed from the previous picture and is also proved in the first section of the Appendix.

$\qquad$  For $\theta<\theta^*$, the rate function $\psi$ is identical with $\chi$. This means that no jump occurs. Conditionally on the event $\{Z_n\geq \exp(\theta n)\}$,  the process first 'just survives with bounded values' until  time $\lfloor t_{\theta}n\rfloor$ ($t_{\theta}\in(0,1)$). Then it grows within a good environment  such that $S_n-S_{\lfloor t_{\theta}n\rfloor}\approx \theta n$ (see Figure \ref{quatre} a)). When $\theta$ increases, the survival period decreases whereas the geometric growth rate of the process remains constant
and is equal to $\theta^*$. 

$\qquad$ For $\theta^*\leq\theta\leq \theta^{\dagger}$, $\psi$ is equal to $\Lambda$. Thus, conditionally on the large deviation event, the process grows exponentially (respectively linearly at the logarithmic scale) from the beginning to the end (see  Figure \ref{quatre} b)). This exceptional growth is due
to a favorable environment such that $S_n\approx \theta n$. 

$\qquad$ For $\theta>\theta^{\dagger}$, the trajectory associated with begins now with a jump : $Z_1\approx \exp(sn)$. Then
it follows an exponential growth which corresponds to a favorable environment $S_n\approx (\theta-s)n$ (see  Figure \ref{quatre} c)). When $\theta$ increases, the initial jump increases whereas the rate of the exponential growth is still equal to $\theta^\dagger$. \\ 

$\qquad$ The case $\theta^{\dagger}=0$ corresponds to  $\psi(\theta)=\gamma+\beta\theta$. Here the good strategy 
consists in just surviving until the end and in one of the prelast generations, one individual has $\exp(\theta n)$-many offsprings (see  Figure \ref{quatre} d)).\\
Finally, we note that in the case $0<\theta=\theta^*=\theta^\dagger$, the best strategy is no longer unique. Indeed, for any $t\in(0,1]$, there exists  $s \in [0,\theta]$ such that all the following trajectories have the same cost. First, the process remains positive and bounded until time $\lfloor tn\rfloor$ (survival period), then it jumps to $\exp(sn)$  and  grows exponentially with a constant rate
 (see Figure \ref{graphlog}). \\

\begin{fig} Representation of $t\in[0,1]\rightarrow f_{\theta}(t)$  in the strongly subcritical case for $\theta_1<\theta_2<\theta_3=\theta^*<\theta_4<\theta_5=\theta^\dagger<\theta_6<\theta_7$.
\label{pathg7}
\begin{center}
$\includegraphics[scale=0.45]{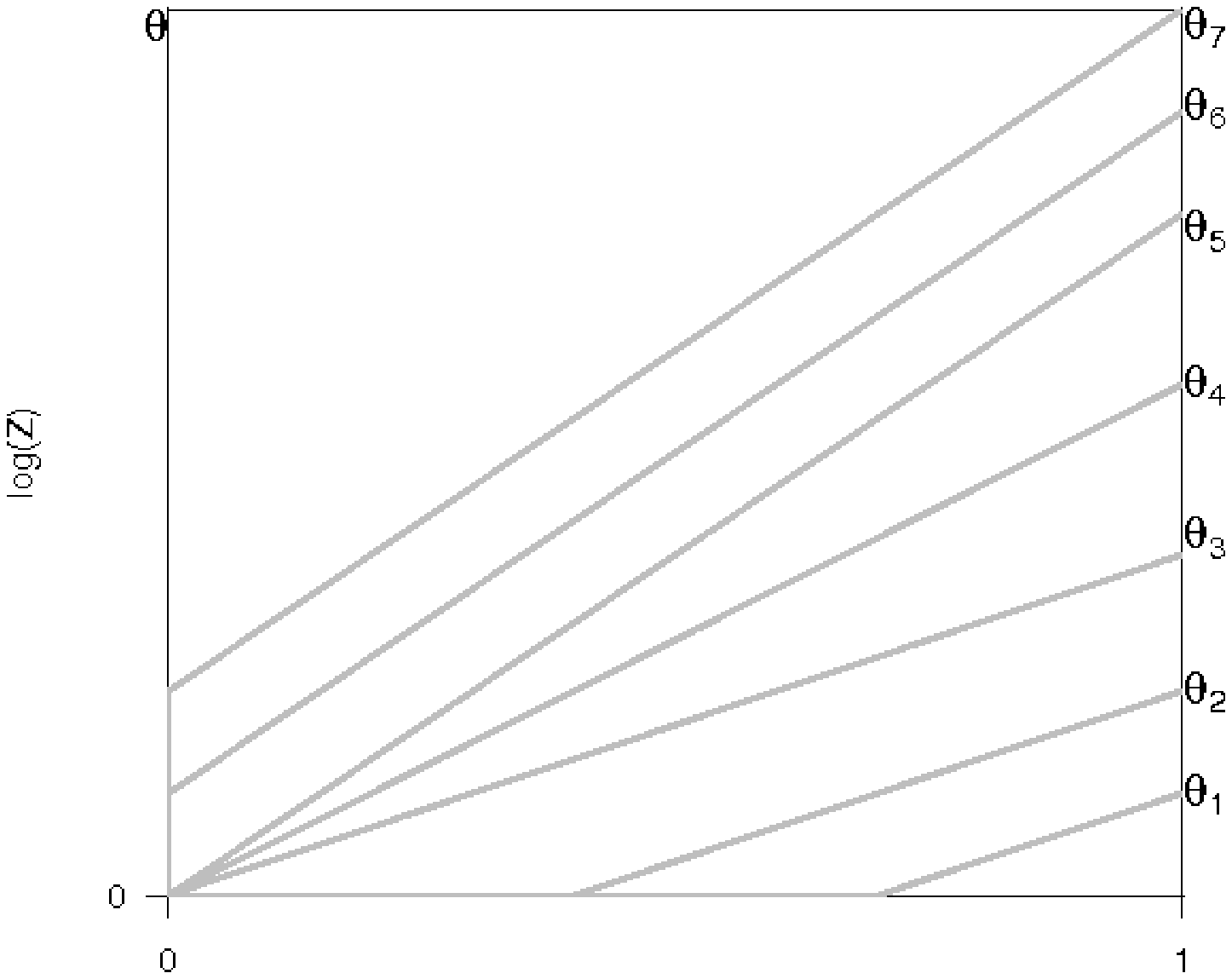}$
\end{center} 
\end{fig}


{\bf Notations:} Unless otherwise is specified, we start the branching process from one single individual and denote by $\P$ the probability associated with.  We denote  by $\P_k$ the probability when the initial size of the population is equal to $k$. Large deviations results actually do not depend on the initial number of individuals if this latter is fixed (or bounded).  

In the whole paper, we  denote by $\Pi:=(f_1,f_2,\ldots)$ the complete environment. 

For simplicity of notations, we are using several times $\leq _c$ to indicate that the inequality holds up to some multiplicative constant (which does not depend on any variable).\\

{\bf Acknowledgements:} {\sl The authors are grateful to G\"otz Kersting
and Julien Berestycki for fruitful discussions. The research was supported in part by the German Research Foundation (DFG), Grant 31120121 and by ANR Manege. Moroever this research benefited from the support of the Chair Modélisation Mathématique et biodiversité VEOLIA-Ecole Polytechnique-MNHN-F.X.}

\section{Proof of the upper bound of Theorem 1}
For the proof of the upper bound of Theorem \ref{theo1}, we need the following result. It ensures that exceptional growth of the population can at least be achieved thanks to some suitable good environment sequences,  whose probability decreases exponentially following the rate function of the random walk $(S_n : n \in\N)$. 
 This result generalizes Proposition 1 in \cite{bansaye08} for an exponential initial number of individuals. 
  With a slight abuse, we denote below by $\exp(s n)$ the initial number of individuals instead of the integer part of $\exp(s n)$. 
\begin{prop}
 Under assumption $\mathcal{H}(\beta)$, for all $\theta\geq 0$ and $0\leq s\leq \theta$,
\begin{eqnarray}
 \limsup_{n\rightarrow\infty} -\frac{1}{n} \log \P_{\exp(s n)}(Z_n \geq \exp(\theta n)) &\leq& \Lambda((\theta-s)+) \ . \nonumber 
\end{eqnarray}
\label{prop1}
\end{prop}
\begin{proof}
For every $\theta'> 0$, we recall that 
$$\Lambda(\theta')=\sup_{\lambda \geq 0}\{\lambda \theta'- \log \E[\exp(\lambda X)]\}.$$
$\qquad$ First, we assume  that $\E[\exp(\lambda X)]<\infty$ for every $\lambda\geq 0$. Then the derivative   of $\lambda\rightarrow \mathbb{E}[\exp(\lambda X)]$ exists for every $\lambda\geq 0$ and the supremum is reached in   $\lambda=\lambda_{\theta'}$ such that
$$\theta'= \frac{\E[X\exp(\lambda_{\theta'} X)]}{\E[\exp(\lambda_{\theta'}X)]}.$$

Following classical large deviations methods and more specifically  \cite{bansaye08}, we  introduce the probability $\w{\P}$  defined by
$$\w{\P} ( X  \in  \d x) =  \frac{\exp(\lambda_{\theta'}x)}{\E[ \exp(\lambda_{\theta'}X)]}\P(X\in \d x).$$
Under this new probability,  $(S_n:n\in\N)$ is a random walk with drift $\w{\E} [X]=\theta'>0$ and $Z_n$ is  a supercritical BPRE. \\

For all $n\geq 1$, $\theta \in [0,\theta')$ and $\epsilon>0$,
\Bea
&& \P_{\exp(sn)}\big(Z_n\geq \exp([\theta+s]n)\big) \\
&& \quad \geq  \P_{\exp(sn)}\big(Z_n\geq \exp([\theta+s]n); \quad S_n\leq (\theta'+\epsilon)n\big) \\
&& \quad = \E[\exp(\lambda_{\theta'}X)]^n\w{\E}_{\exp(sn)}\big[\exp(-\lambda_{\theta'}S_n)\ind_{\{S_n\leq (\theta'+\epsilon)n, \ Z_n\geq \exp([\theta+s]n)\}}\big] \\
&& \quad \geq \exp\big(n \big[\log(\E[\exp(\lambda_{\theta'}X )]-\lambda_{\theta'}(\theta'+\epsilon)\big]\big) \w{\P}_{\exp(sn)}\big(Z_n\geq \exp([\theta+s]n), \ S_n\leq (\theta'+\epsilon)n \big) \\
&& \quad \geq  \exp(n[-\Lambda(\theta')-\lambda_{\theta'}\epsilon])\big[\w{\P}_{\exp(sn)}\big(Z_n\geq \exp([\theta+s]n)\big)
-\w{\P}\big(S_n> (\theta'+\epsilon)n \big)\big].
\Eea
As $\w{\P}\big(S_n> (\theta'+\epsilon)n \big)\rightarrow 0$ when $n\rightarrow \infty$, we  just need to prove that 
\be
\label{qtpos}
\liminf_{n\rightarrow \infty} \w{\P}_{\exp(sn)}\big(Z_n\geq \exp([\theta+s]n)\big)>0.
\ee
so that we can conclude the proof by letting $\epsilon\rightarrow 0, \theta'\rightarrow \theta$. \\

Relation (\ref{qtpos}) results from the fact that under $\w{\P}$ the population $Z_n$ starting from one single individual grows as $\exp(S_n)\asymp n\theta'$ on the non-extinction event. More precisely,  individuals of the initial population are labeled and the number of descendants in generation $n$ of individual $i$ is denoted by $Z_n^{(i)}$. Introduce then the 'success' probability 
$p_n$:  
$$p_n=\P_1(Z_n\geq N\exp(n\theta  ) \ \vert \ \Pi ) \quad \mbox{a.s}.$$
Then, conditionally on $\Pi$,  for $N\geq 1$, the number of initial individuals
whose number of descendants  in  generation $n$ is larger than $N\exp(n\theta)$,
$$N_n:=\#\{ 1\leq i\leq \exp(sn) : Z_n^{(i)}\geq N\exp(n\theta)\},$$
follows a binomial distribution
of parameters $(\exp(sn),p_n)$.
Moreover, as $\E[N_n \ \vert  \ \Pi]=e^{sn}p_n$ a.s.,
$$\w{\P}_{\exp(sn)}\big(Z_n\geq \exp([\theta+s]n)\big)\geq \w{\P}_{\exp(sn)}\big( N_n\geq \exp(sn)/N)\geq 
 \w{\P}_{\exp(sn)}\Big( N_n\geq \frac{\E[N_n \ \vert  \ \Pi]}{Np_n}\Big).$$
Using the classical inequality due to Paley and Zygmund for $r\in [0,1]$ (see e.g. \cite{kallenberg} page 63),
\begin{eqnarray}
 \P(Y\geq r\E[Y])&\geq& (1-r)^2 \frac{\E[Y]^2}{\E[Y^2]} \ ,\label{pz}
\end{eqnarray}
and adding that  $\E[N_n^2\ \vert  \ \Pi]=e^{2sn}p_n^2+e^{sn}p_n(1-p_n)$ a.s., we get 
$$ \w{\P}_{\exp(sn)}\Big( N_n\geq \frac{\E[N_n \ \vert  \ \Pi]}{Np_n}  \ \Big| \ \Pi \Big)\geq 
\Big[1- 1\wedge \frac{1}{Np_n}\Big]^2\frac{\E[N_n \ \vert \ \Pi]^2}{\E[N_n^2\ \vert \ \Pi]}\geq \frac{\Big[1- 1\wedge \frac{1}{Np_n}\Big]^2}{1+\frac{e^{-sn}}{p_n}} \ \ \mbox{a.s.}$$

Now, we use that under assumption $\mathcal{H}(\beta)$, 
$$\w{\E}\big[\sum_{k\in\N} k^s\P(L=k \ \vert \ f )/m\big]\leq \w{\E}\big[\sum_{k\in\N} k^s\P(L=k \ \vert \ f, \ L>0)\big]<\infty,$$
for every $1<s<\beta$. So Theorem 3 in \cite{GuivL} ensures that for every $N\in\N$,
$$\w{\E}[p_n]=\w{\P}_1(Z_n\geq N\exp(\theta n))\stackrel{n\rightarrow \infty}{\longrightarrow } \w{\P}_1(\forall n \in \N: Z_n>0) >0 .$$
As the right hand side  does  not depend on $N\geq 1$,  we have for $N$ large enough 
$$\delta:=\liminf_{n\rightarrow \infty} \w{\P}(p_n\geq 2/N)>0$$
and get 
$$\liminf_{n\rightarrow \infty}\P_{\exp(sn)}\big(Z_n\geq \exp([\theta+s]n)\big)\geq  \liminf_{n\rightarrow \infty} \w{\E}\Big[\frac{\big[1- 1\wedge 1/Np_n\big]^2}{1+N/2}\Big] \geq \frac{\delta (1-1/2)^2}{1+N/2}>0,$$
which proves (\ref{qtpos}) and ends up the proof when $\E[\exp(\lambda X)]<\infty$ for every $\lambda\geq 0$. The general case follows by a standard approximation argument (see e.g.  \cite{BK09} pages 10/11). 
\end{proof}
\begin{proof}[Proof of the upper bound in Theorem \ref{theo1}]
The proof amounts now to exhibit good trajectories which realize the large deviation event $\{Z_n\geq \exp(\theta n)\}$.  For every $t\in (0,1)$ and $s\in[0,\theta]$, by Markov property,
$$
\P(Z_n\geq \exp(\theta n))
\geq  \P(Z_{[tn]}>0)\P(Z_1\geq \exp(sn))\P_{\exp(sn)}(Z_{n-[tn]}\geq \exp(\theta n)).
$$
First, by ($\ref{lem31}$),
$$-\frac{1}{tn}\log(\P(Z_{[tn]}>0))\stackrel{n\rightarrow \infty}{\longrightarrow } \gamma \ .$$
Second, using that  that $\log(\P(Z_1>z))/\log(z) \stackrel{z\rightarrow \infty}{\longrightarrow } -\beta$,
we have
$$-\frac{1}{n}\log(\P(Z_1\geq \exp(sn)))\stackrel{n\rightarrow \infty}{\longrightarrow }s\beta \ .$$

Finally, by Proposition \ref{prop1},  we get that
$$\limsup_{n\rightarrow \infty} -\frac{1}{(1-t)n}\log(\P_{\exp(sn)}(Z_{n-[tn]}\geq \exp(\theta n)))\leq  \Lambda((\theta-s)/(1-t)+)$$
since
$$\P_{\exp(sn)}(Z_{n-[tn]}\geq \exp(\theta n))=\P_{\exp(s/(1-t).(1-t)n)}(Z_{n-[tn]}\geq \exp(n(1-t)\theta/(1-t))).$$
Combining the first inequality and the last three limits ensures that
$$\limsup_{n\rightarrow \infty}  -\frac{1}{n}\log(\P(Z_n\geq \exp(\theta n)))\leq \inf_{t\in[0,1],  s\in [0,\theta]} \Big\{t\gamma+\beta s+(1-t)\Lambda((\theta-s)/(1-t)+)\Big\}.$$
As convex nonnegative function, $\Lambda$ has at most one jump (to infinity). Thus the above infimum is $\psi(\theta)$. To see this, we only have to consider the jump point. Say, there are $s_\theta\in[0,\theta]$ and $t_\theta\in[0,1)$ such that 
$$ t_\theta\gamma+\beta s_\theta+(1-t_\theta)\Lambda((\theta-s_\theta)/(1-t_\theta))=\psi(\theta)<\infty$$
and $\Lambda((\theta-s_\theta)/(1-t_\theta)+)=\infty$. Then, as $(\theta-s_\theta)/(1-t_\theta)$ is the only jump point, for any $\epsilon>0$ there is a $\delta>0$ such that
\begin{eqnarray}
\psi(\theta)-\epsilon&\leq& t_\theta\gamma+\beta (s_\theta-\delta)+(1-t_\theta)\Lambda((\theta-s_\theta-\delta)/(1-t_\theta)+) \nonumber \\
&=&  t_\theta\gamma+\beta (s_\theta-\delta)+(1-t_\theta)\Lambda((\theta-s_\theta-\delta)/(1-t_\theta)). \nonumber
\end{eqnarray}
Now letting $\epsilon\rightarrow 0$ proves the result and thereby the upper bound of Theorem  \ref{theo1}.
\end{proof}

\section{Proof of the lower bound of Theorem 1 for $\beta \in (1,2]$}
\label{beta12}
We introduce the minimum of the associated random walk up to time $n$: 
\begin{eqnarray}
 M_n &:=& \min_{0\leq k\leq n} S_k \ \ \ \mbox{a.s.} \nonumber 
\end{eqnarray}
Using that $\mathbb{P}(Z_n>0|\Pi) \leq \E[Z_n|\Pi]=\exp(S_n)$  and $\mathbb{P}(Z_n>0|\Pi)$ decreasing a.s., we get
the following classical inequality (see e.g. \cite{birkner05})
\begin{eqnarray}
\mathbb{P}(Z_n>0|\Pi) \leq e^{M_n} \ \ \ \mbox{a.s.} \label{upbon} 
\end{eqnarray}
Actually,  the above estimate  gives the correct exponential decay rate (see e.g. \cite{birkner05}):
\begin{eqnarray}
\gamma &=& -\lim_{n\rightarrow \infty}\frac{1}{n} \log \P(Z_n>0)= - \lim_{n\rightarrow\infty} \frac{1}{n} \log \E\big[e^{M_n}\big]. \nonumber
\end{eqnarray}
In Lemma \ref{gam2}, the above relation is generalized and proved rigorously under assumption $\mathcal{H}(\beta)$.\\
$\qquad$ For the proof of the lower bound of the main theorem, we need the following key bound for the tail probability of $Z_n$.
\begin{theo}
 Under assumption $\mathcal{H}(\beta)$  for some $\beta\in(1,2]$, there exist a constant $0<c<\infty$ and a 
positive nondecreasing and slowly varying function $\Upsilon$ such that for all $k\geq 1$ and $n\geq 1$,
\begin{eqnarray}
 \P(Z_n>k|\Pi) &\leq& c n^{\lceil\beta\rceil}\Upsilon(n^{2/(\beta-1)} e^{-M_n}k) e^{M_n}   (e^{S_n-M_n}/k)^{\beta}  \quad a.s.  \nonumber
\end{eqnarray}
\label{theobond}
\end{theo}
Let us explain briefly this result. The probability to survive until time $n$ evolves as $\exp(M_n)$, nice environment sequences correspond to large values of $(S_n-M_n)$ and high reproduction of the initial individual gives the last term  $k^{-\beta}$. Conditionally on the environment sequence and the survival of the process, the growth of the process follows $\exp(S_n-M_n)$ : this corresponds to 'best period' in time for the  growth of the process. Thus, this theorem
 essentially says that conditionally on $Z_n>0$,  the tail distribution of $Z_n/e^{S_n-M_n}$ is at most  polynomial  with exponent $-\beta$. 
 
Recalling  that $\Pi=(f_1,f_2,...)$ and $f_n(s) $ is  probability generating function of the offspring distribution of an individual in generation $n-1$, we have 
\be
 f_{0,n}(s):=\sum_{k=0}^{\infty} s^k \P(Z_n=k|\Pi)=\E[s^{Z_n}|\Pi], \quad \mbox{a.s.} \quad (0\leq s\leq 1). \label{genf}
\ee
For the proofs, it is suitable to work with an alternative expression, namely for every $n\geq 1$,
$$g_n(s):=\frac{1-f_n(s)}{1-s} \quad \mbox{a.s.} \quad (0\leq s\leq 1)$$
and 
\begin{eqnarray}
g_{0,n}(s)&:=&\sum_{k=0}^{\infty} s^k \P(Z_n>k|\Pi)= \frac{1-f_{0,n}(s)}{1-s} \quad \mbox{a.s.} \quad (0\leq s\leq 1).\label{defg2} 
\end{eqnarray}
Moreover we need the following auxiliary function defined for every $\mu\in(0,1]$ by
\be
 h_{\mu,k}(s) := \frac{1}{(1-f_k(s))^{\mu}}-\frac{1}{(f'_k(1)(1-s))^{\mu}}=\frac{g_k(1)^{\mu}-g_k(s)^{\mu}}{(g_k(1) g_k(s) (1-s))^{\mu}}  \quad \mbox{a.s.} \quad (0\leq s\leq 1). \label{hdef}
\ee

Finally, we define for all $1\leq k\leq n$,
\Bea 
U_k &:=& \left( f'_1(1) \cdots f'_k(1)\right)^{-1} \ = \ f'_{0,k}(1)^{-1} \ = \ e^{-S_k}, \\
 f_{k,n} &:=& f_{k+1}\circ f_{k+2}\circ\cdots \circ f_{n}, \quad 0\leq k< n; \ f_{n,n}=id \quad \mbox{a.s.}
\Eea

By a telescope summation argument similar to \cite{kersting00}, we have
\begin{eqnarray}
 \frac{1}{(1-f_{0,n}(s))^{\mu}} &=& \frac{U^{\mu}_0}{(1-f_{0,n}(s))^{\mu}} \nonumber \\
&=&\frac{U^{\mu}_n}{(1-f_{n,n}(s))^{\mu}} + \sum_{k=0}^{n-1} \left(\frac{U^{\mu}_{k}}{(1-f_{k,n}(s))^{\mu}} -\frac{U^{\mu}_{k+1}}{(1-f_{k+1,n}(s))^{\mu}}\right) \nonumber \\
&=&\frac{U^{\mu}_n}{(1-s)^{\mu}} + \sum_{k=0}^{n-1} U^{\mu}_k \left(\frac{1}{(1-f_{k+1}(f_{k+1,n}(s)))^{\mu}} -\frac{1}{(f'_{k+1}(1)(1-f_{k+1,n}(s)))^{\mu}}\right) \nonumber \\ 
&=&\frac{U^{\mu}_n}{(1-s)^{\mu}} + \sum_{k=0}^{n-1} U^{\mu}_k h_{\mu,k+1}(f_{k+1,n}(s)), \quad s\geq 0. \label{tele2}
\end{eqnarray}
$\newline$

\begin{proof}[Proof of Theorem \ref{theobond}]
In the same vein as \cite{BK09}, we are   obtaining an upper bound for $\P(Z_n>z|\Pi)$ from the divergence of 
$g_{0,n}'(s)=\sum_{j=0}^{\infty} j \P(Z_n>j|\Pi) s^{j-1}$  as $s\rightarrow 1$. In that purpose, we use
 (\ref{tele2}) for $\mu=\beta-1$ and get 
\begin{eqnarray}
 g_{0,n}(s) &=& \Big(U_n^{\beta-1} + (1-s)^{\beta-1}\sum_{k=0}^{n-1} U^{\beta-1}_k h_{\beta-1,k+1}(f_{k+1,n}(s))\Big)^{-1/(\beta-1)} \quad  (0\leq s\leq 1) \quad \t{a.s.} \nonumber
\end{eqnarray}
Then we calculate the first derivative of $g_{0,n}$: 
\begin{eqnarray}
 &&g'_{0,n} (s)  \nonumber \\
&=& -(\beta-1)^{-1} \Big(U_n^{\beta-1} + (1-s)^{\beta-1}\sum_{k=0}^{n-1} U^{\beta-1}_k h_{\beta-1,k+1}(f_{k+1,n}(s))\Big)^{-1-1/(\beta-1)} \nonumber\\
&& \qquad 
\times  \Big(-(\beta-1) (1-s)^{\beta-2}\sum_{k=0}^{n-1} U^{\beta-1}_k h_{\beta-1,k+1}(f_{k+1,n}(s)) \nonumber \\ && \qquad \qquad 
+(1-s)^{\beta-1}\sum_{k=0}^{n-1} U^{\beta-1}_k  h'_{\beta-1,k+1}(f_{k+1,n}(s))f'_{k+1,n}(s)\Big) \nonumber \\
&\leq &\frac{\sum_{k=0}^{n-1} U^{\beta-1}_k \Big(h_{\beta-1,k+1}(f_{k+1,n}(s))-(\beta-1)^{-1}  h'_{\beta-1,k+1}(f_{k+1,n}(s))f'_{k+1,n}(s)(1-s)\Big)}{U_n^{\beta} (1-s)^{2-\beta}} \quad \label{gder}
\end{eqnarray}

Now Lemma \ref{bound2} in the appendix ensures that there exists  $c>0$ such that for every $s\in[0,1)$, 
\begin{eqnarray}
 h_{\beta-1,k}(s) &\leq& c \Upsilon(1/(1-s)), \label{boundhtilde}  \\
 -h'_{\beta-1,k}(s) &\leq& c \Upsilon(1/(1-s)) /(1-s) \quad \mbox{a.s.} \label{boundhptilde}
\end{eqnarray}
Moreover, using (\ref{tele2}), Lemma \ref{bound2} in the appendix for $0<\mu<\beta-1$ and
$U_k\leq \exp(-M_n)$ for every $0\leq k\leq n$, there exists a $c\geq 1$ such that for every $s\in[0,1)$,
\begin{eqnarray}
 \frac{1}{(1-f_{k+1,n}(s))^{\mu}} &\leq& \frac{e^{-\mu M_n}}{(1-s)^\mu} + n \ c \ e^{-\mu M_n}\leq ce^{-\mu M_n}(n+1)/(1-s)^{\mu} \quad  \mbox{a.s.} \nonumber
\end{eqnarray}
Combining this inequality with (\ref{boundhtilde}) ensures that there exists  $c>0$ such that
\begin{eqnarray}
 h_{\beta-1,k+1}(f_{k+1,n}(s)) &\leq&  \ c \Upsilon\big( (n+1)^{1/\mu} e^{-M_n} (1-s)^{-1}\big) \quad (0\leq s<1) \quad \mbox{a.s.} \nonumber 
\end{eqnarray}
Moreover,  $f_{k+1,n}(s) \leq 1-f'_{k+1,n}(s) (1-s)$ by convexity of $f_{k+1,n}$ and (\ref{boundhptilde}) ensures that
\begin{eqnarray}
-h_{\beta-1,k+1}'(f_{k+1,n}(s))f'_{k+1,n}(s)(1-s) & \leq & c \ f'_{k+1,n}(s)(1-s)  \Upsilon\big(1/(1-f_{k+1,n}(s))\big) \frac{1}{1-f_{k+1,n}(s)} \nonumber \\
&\leq&   c \ \Upsilon\big( (n+1)^{1/\mu} e^{-M_n} /(1-s)\big) \quad (0\leq s< 1) \quad \mbox{a.s.} \nonumber 
\end{eqnarray}
Using the two last  estimates with $\mu=(\beta-1)/2$ together in (\ref{gder}) yields
\begin{eqnarray}
 g_{0,n}'(s) &\leq& c\frac{\  n\ e^{-(\beta-1) M_n} \Upsilon\big( (n+1)^{2/(\beta-1)} e^{-M_n} (1-s)^{-1}\big)}{U_n^{\beta} (1-s)^{2-\beta}}\quad (0\leq s\leq 1) \quad \mbox{a.s.} \nonumber
\end{eqnarray} 
$\qquad$ Moreover for all $k\geq 1$ and $s\in[0,1]$,
\begin{eqnarray}
g'_{0,n}(s) &\geq& \sum_{j=k/2}^{k} j \P(Z_n>j|\Pi) s^{j-1} \nonumber \\
&\geq& s^{k} \frac{k^2}{2} \P(Z_n>k|\Pi).\label{derivativeg}
\end{eqnarray}
By letting  $s=1-1/k$ in the two last inequalities, we get
\begin{eqnarray}
\Big(1-\frac{1}{k}\Big)^k \frac{k^2}{2} \P(Z_n>k|\Pi) &\leq& c \frac{n \  e^{-(\beta-1) M_n} k^{2-\beta} \ \Upsilon\big(k(n+1)^{2/(\beta-1)}e^{-M_n}\big)}{U_n^{\beta}},\nonumber
\end{eqnarray}
 which ends up the proof since $U_n=\exp(-S_n)$.
 \end{proof}
$\newline$

For the proof of the lower bound of the Theorem \ref{theo1},  we also need the  following characterization of the 'survival cost' $\gamma$:
\begin{lemma}
Under assumption $\mathcal{H}(\beta)$, for all $\theta\geq 0$, $b>0$ and $\Upsilon$  positive nondecreasing and slowly varying at infinity, 
\begin{eqnarray}
\gamma &=& - \lim_{n\rightarrow\infty} \frac{1}{n} \log \E\big[\Upsilon(n^b e^{\theta n} e^{-M_n})e^{M_n}\big] \ . \nonumber
\end{eqnarray}
\label{gam2}
\end{lemma}
\begin{proof}[Proof of Lemma \ref{gam2}]
First let $\Upsilon=1$. We use (\ref{tele2}) with some $0<\mu<\beta-1$ 
and  (\ref{eqth}) ensures that
\begin{eqnarray}
  \P(Z_n>0|\Pi) &\geq& \frac{1}{(e^{-\mu S_n} + \sum_{k=0}^{n-1} e^{-\mu S_k} h_{\mu,k+1}(f_{k+1,n}(1)))^{1/\mu}} \nonumber \\
&\geq&  c^{-1} n^{-1/\mu}e^{M_n}. \nonumber  
\end{eqnarray}
For the upper bound, we use (\ref{upbon}) and get 
\begin{eqnarray}
\gamma &=& -\lim_{n\rightarrow \infty}\frac{1}{n} \log \P(Z_n>0)= - \lim_{n\rightarrow\infty} \frac{1}{n} \log \E\big[e^{M_n}\big]. \nonumber
\end{eqnarray}
As $\Upsilon$ is nondecreasing, 
\be
\gamma \geq \limsup_{n\rightarrow\infty} -\frac{1}{n} \log \mathbb{E}[\Upsilon(n^b e^{\theta n} e^{-M_n}) e^{M_n}]. \nonumber
\ee
For the converse inequality, we use that $\E\big[e^{t M_n}\big]$ is nonincreasing in $n$ to define
\begin{eqnarray}
 \xi(t) &:=& -\lim_{n\rightarrow\infty} \frac{1}{n} \log\E\big[e^{t M_n}\big] \ . \nonumber
\end{eqnarray}
We note that  $\xi(t)\geq 0$ and by Lemma V.4 in \cite{hollander}, $\xi(t)$ is finite and
convex. So $\xi$ is continuous. \\
Now by properties of slowly varying sequences (see \cite{BGT}, proposition 1.3.6, page 16), for any $\delta>0$, $x^{-\delta}\Upsilon(x)\rightarrow 0$ as $x\rightarrow\infty$ (see appendix) and 
 $$-\lim_{n\rightarrow\infty} \frac{1}{n} \log \E\big[\Upsilon(n^b e^{\theta n}e^{-M_n})e^{M_n}\big] \geq -\delta \theta-\lim_{n\rightarrow\infty} \frac{1}{n} \log \E\big[e^{(1+\delta)M_n}\big].$$ 
Letting $\delta\rightarrow0$ and using continuity of $\chi$, this ends up the proof. 
\end{proof}
$\newline$
\begin{proof}[Proof of the lower bound of Theorem 1] First, we recall the following classical large deviation inequality:
\be
\label{ineqLD}
\P(S_n\geq \theta n)\leq e^{-\Lambda(\theta)n}
\ee
and we define the first time $\tau_n$ when the random walk $(S_i : i\leq n)$ reaches its minimum 
value on $[0,n]$:
$$\tau_n:=\inf\{0\leq k \leq n: \ S_k=M_n\}.$$

We decompose the probability of having an extraordinarily large population according to $S_n-M_n$.
\begin{eqnarray}
\label{decomp}
\P(Z_n\geq e^{\theta n}) &=& \P(Z_n\geq e^{\theta n}, S_n-M_n\geq \theta n)+ \E[\P(Z_n\geq e^{\theta n}|\Pi); S_n-M_n<\theta n]. 
\end{eqnarray}
$\qquad$ The asymptotic of the first term can be found using (\ref{ineqLD}) (see \cite{BK09}): 
\Bea
\P(Z_n\geq e^{\theta n},\  S_n-M_n\geq \theta n) &\leq &\sum_{i=1}^n \P(Z_i>0, \P(S_n-S_{i}\geq \theta n)\\
&\leq &\sum_{i=1}^n \P(Z_i>0)\exp(-(n-i)\Lambda(\theta n/(n-i))).
\Eea
This ensures that
\be
 \liminf_{n\rightarrow\infty} -\frac{1}{n} \log \P(Z_n\geq e^{\theta n}, S_n-M_n\geq \theta n)\geq  \chi(\theta), \label{0806}
\ee 
where
\begin{eqnarray} 
 \chi(\theta) &=& \inf_{0<t\leq 1} \big\{t\gamma + (1-t) \Lambda(\theta/(1-t))\big\}. \nonumber
\end{eqnarray}
$\qquad$ For the second term, we use Theorem \ref{theobond} and the Markov property for $(S_n: n\geq 0)$:
\begin{eqnarray}
&& \E[\P(Z_n\geq e^{\theta n}|\Pi); S_n-M_n<\theta n] \nonumber\\
&& \quad \leq c \ n^{\lceil \beta \rceil} \ \E\Big[\Upsilon(n^{2/(\beta-1)} e^{-M_n} e^{\theta n})
e^{-M_n}e^{\beta(S_n-M_n-\theta n)}; S_n-M_n<\theta n\Big] \nonumber \\
&& \quad =c \ n^{\lceil \beta \rceil}\ \sum_{k=0}^n \E\Big[\Upsilon(n^{2/(\beta-1)} e^{-S_k} e^{\theta n})
e^{S_k}e^{\beta (S_n-S_k-\theta n)}; S_n-M_n<\theta n, \tau_n=k\Big] \nonumber \\
&& \quad \leq c \ n^{\lceil \beta\rceil} \ \sum_{k=0}^{n} \E[\Upsilon(n^{2/(\beta-1)} e^{-S_k} e^{\theta n})e^{S_k};\tau_k=k] \E[e^{-\beta(\theta n-S_{n-k})};S_{n-k}<\theta n, M_{n-k}\geq 0] \nonumber 
\end{eqnarray}
Let $\epsilon=1/n^2$ and $m_{\epsilon}=\lceil \theta/\epsilon\rceil$. Using that
$$\E[\Upsilon(n^{2/(\beta-1)} e^{-S_k} e^{\theta n})e^{S_k};\tau_k=k] =\E[\Upsilon(n^{2/(\beta-1)} e^{-M_k} e^{\theta n})e^{M_k},\tau_k=k] \leq \E[\Upsilon(n^{2/(\beta-1)} e^{-M_k} e^{\theta n})e^{M_k}]$$
and we deduce from  (\ref{ineqLD}) that 
\begin{eqnarray}
  &&\E[\P(Z_n\geq e^{\theta n}|\Pi); S_n-M_n<\theta n] \nonumber \\
&\leq& c \ n^{\lceil \beta\rceil} \ \sum_{k=1}^{n} \E[\Upsilon(n^{2/(\beta-1)} e^{-M_k} e^{\theta n})e^{M_k}]  \sum_{j=0}^{m_{\epsilon}} e^{-\beta(\theta-(j+1)\epsilon)n} \P\big(S_{n-k}\in [n j \epsilon, n (j+1) \epsilon), M_{n-k}\geq 0\big) \nonumber \\
&\leq & c \ n^{\lceil \beta\rceil}  \ \sum_{k=1}^{n} \E[\Upsilon(n^{2/(\beta-1)} e^{-M_k} e^{\theta n})e^{M_k}]  \sum_{j=0}^{m_{\epsilon}} e^{-\beta(\theta-(j+1)\epsilon)n} e^{-\Lambda(j\epsilon n/(n-k))(n-k)} \nonumber \\
&\leq& c \ \theta \ n^5 \  \sup_{0< t\leq 1, 0\leq s \leq \theta} \left\{\E\Big[\Upsilon(n^{2/(\beta-1)} e^{-M_{\lfloor t n\rfloor}} e^{\theta n})e^{M_{\lfloor t n\rfloor}}\Big]\cdot e^{-(\beta s+(1-t) \Lambda((\theta-s)/(1-t)))n}\right\}.\nonumber
\end{eqnarray}
Together with Lemma \ref{gam2}, this yields
\begin{eqnarray}
 \liminf_{n\rightarrow\infty}- \frac{1}{n}\log \E[\P(Z_n\geq e^{\theta n}|\Pi); S_n-M_n<\theta n]&\geq& \psi(\theta), \nonumber
\end{eqnarray}
where
\begin{eqnarray}
 \psi( \theta) &=& \inf_{0< t\leq 1, 0\leq s \leq \theta} \Big\{\gamma t + \beta s + (1-t)\Lambda((\theta-s)/(1-t))\Big\}.\nonumber 
\end{eqnarray}
Combining this inequality with (\ref{decomp}) and (\ref{0806}) gives
\be
 \liminf_{n\rightarrow\infty}-\frac{1}{n} \log \P(Z_n\geq e^{\theta n})\geq \min\{\chi(\theta); \psi(\theta)\}.\nonumber 
\ee
Adding that $\psi(\theta)\leq \chi(\theta)$ since the infimum is considered on a larger set for $\psi$ than for $\chi$, we get 
\be
 \limsup_{n\rightarrow\infty}-\frac{1}{n} \log \P(Z_n\geq e^{\theta n})\geq \psi(\theta), \nonumber
\ee
which proves the lower bound of Theorem $1$. 
\end{proof}

\section{Adaptation of the proof of the lower bound  for $\beta>2$}
\label{beta2}
$\qquad$ First, Lemma \ref{gam2} still holds for $\beta>2$ by following the same proof. Indeed,  using  
 (\ref{tele2}) for  $\mu=1$ together with Lemma \ref{bound2} given in the appendix ensures that
\begin{eqnarray}
  \P(Z_n>0|\Pi)  =  1-f_{0,n}(0)&\geq& \frac{1}{e^{-S_n} + \sum_{k=0}^{n-1} e^{-S_k} h_{k+1}(f_{k+1,n}(0))} \geq n^{-1} \ c^{-1} e^{M_n}. \nonumber  
\end{eqnarray}

The main difficulty is to obtain an equivalent of Theorem \ref{theobond}. For this, we need
 to calculate higher order derivatives of $g_{0,n}$ and the upper bound on the  tail probability of $Z_n$ contains an additional term:
\begin{theo}
 Under assumption $\mathcal{H}(\beta)$ for some $\beta>2$, there are a constant $0<c<\infty$ and a positive nondecreasing slowly varying function $\Upsilon$ such that for every $k\geq 1$,
\begin{align*}
 P(Z_n>k|\Pi) &\leq& c\ e^{S_n} n^\beta  \Upsilon(n^{2}e^{-M_n} k) \max\big\{k^{-\beta} e^{(\beta-1)(S_n-M_n)};k^{-\lceil\beta\rceil-1} e^{\lceil\beta\rceil(S_n-M_n)}\big\}\quad a.s.  \nonumber
\end{align*}
\label{theobond2}
\end{theo}

For the proof, we use  the functions 
$$
 h_{k}(s) = \frac{1}{(1-f_k(s))}-\frac{1}{f'_k(1)(1-s)}=\frac{g_k(1)-g_k(s)}{g_k(1) g_k(s) (1-s)}  \quad \mbox{a.s.} \quad (0\leq s< 1) 
$$
and
\be
\label{defH}
     H(s)=\sum_{k=0}^{n-1} U_k h_{k+1}(f_{k+1,n}(s)) \quad \mbox{a.s.} \quad (0\leq s< 1). 
\ee
Then  (\ref{tele2}) with $\mu=1$ gives
\begin{eqnarray}
 g_{0,n}(s)^{-1} &=&\frac{1-s}{1-f_{0,n}(s)} = U_n+ (1-s) H(s) \quad \mbox{a.s.} \quad (0\leq s< 1) \nonumber
\end{eqnarray}
and calculating the $l$-th derivative of the above equation, we get for all $l\geq 1$ and $s\in[0,1)$,
\begin{eqnarray}
 \frac{d^l}{ds^l} g_{0,n}(s)^{-1} &=& (1-s) H^{(l)}(s) - l H^{(l-1)}(s)\quad \mbox{a.s.} \quad (0\leq s< 1).\label{23061}
\end{eqnarray}
The rest of the section is organized as follows. First, we prove the following technical lemma which gives useful bounds for power generating series. Then we derive  Theorem
\ref{theobond2}. Finally the main lines of the proof of the lower bound of Theorem \ref{theo1} for $\beta>2$ are explained (following the proof for $\beta\in(1,2]$). For simplicity of notation, we introduce $\leq_c$ which means that the inequality is fulfilled up to a multiplicative constant $c$ which does not depend
on $s$, $k$, $l$ or $\omega$.  
\begin{lemma}
 Under assumption $\mathcal{H}(\beta)$, for every $l\leq \lceil\beta\rceil-1$,
\begin{eqnarray}
 f^{(l)}_{0,n}(1) &\leq_c& n^{l-1} \ e^{S_n} e^{(l-1) (S_n-M_n)} \quad  a.s. \label{2406}
\end{eqnarray}
Moreover the following  estimates   hold a.s. for every $s\in [0,1)$ respectively
for  $l<\lceil\beta\rceil-2$, $l=\lceil\beta\rceil-2$ 	and $l=\lceil\beta\rceil-1$
\begin{eqnarray}
 |H^{(l)}(s)| 
&\leq_c& n^{l}\ e^{l(S_n-M_n)}  \label{2306} \\
 |H^{(l)}(s)| 
&\leq_c&  n^{l} e^{(\lceil\beta\rceil-2)(S_n-M_n)}\nonumber \\
&&+n\Upsilon(n^2 e^{-M_n} (1-s)^{-1})(1-s)^{-(\lceil\beta\rceil-\beta)}e^{-S_n} e^{(\beta-1)(S_n-M_n)} \label{1101} \\
 |H^{(l)}(s)| 
&\leq_c& n^{l} e^{(\lceil\beta\rceil-1)(S_n-M_n)}+n^2\Upsilon(n^2 e^{-M_n} (1-s)^{-1})\ e^{-S_n} e^{\beta(S_n-M_n)} (1-s)^{-(\lceil\beta\rceil-\beta)}\nonumber \\
&&\quad +n\Upsilon(n^2 e^{-M_n} (1-s)^{-1})\ e^{-S_n} e^{(\beta-1)(S_n-M_n)} (1-s)^{-1-(\lceil\beta\rceil-\beta)}. \label{11012}
\end{eqnarray}
\label{lemdif1}
\end{lemma}
\begin{proof}
We prove the Lemma by induction with respect to $l$ and all the following relations hold a.s. for every $s\in[0,1)$. For $l=1$, (\ref{2406})
is trivially fulfilled since $f'_{0,n}(1)=e^{S_n}$. First, we consider $l<\lceil\beta\rceil-2$  and  we assume  that  (\ref{2406}) holds for 
every $i\leq l$. We are first proving that  (\ref{2306}) holds for $l$ and then that (\ref{2406})
holds for $l+1$.

By induction assumptions and  monotonicity of generating functions and its derivatives, for all $i\leq l$ and $s\in[0,1]$, 
\begin{eqnarray}
 f_{k+1,n}^{(i)}(s)  \ \leq \ f_{k+1,n}^{(i)}(1) & \leq_c  & n^{i-1} \ e^{S_n} \ e^{(i-1)(S_n-S_k-\min_{j=k,..,n} \{S_{j}-S_k\})}\nonumber \\
&\leq_c& n^{i-1}\ e^{S_n} \ e^{(i-1)(S_n-M_n)}. \label{2611}
\end{eqnarray}
Lemma \ref{lemder} given in the appendix ensures that (see Lemma \ref{lemder}) for the definition of $u_{j,l}$) 
\begin{eqnarray}
 \Big|\frac{d^l}{ds^l} h_{k+1}(f_{k+1,n}(s))\Big| &=& \Big|\sum_{j=1}^{l} h_{k+1}^{(j)}(f_{k+1,n}(s)) u_{j,l}(s) \Big| \nonumber
\end{eqnarray}
and using (\ref{2611})
\begin{eqnarray}
u_{j,l}(s) &\leq_c& n^{l-j} \ e^{jS_n} \ e^{(l-j)(S_n-M_n)} \ \leq_c \ n^{l-1} \ e^{S_n} \ e^{(l-1)(S_n-M_n)}. \nonumber
\end{eqnarray}
By Lemma \ref{lem3} also given in the appendix, for $j< \lceil\beta\rceil-2$, the derivatives $h^{(j)}_k$ are bounded by a constant that does not depend on $\omega$. 
Thus
\begin{eqnarray}
 \Big|\frac{d^l}{ds^l} h_{k+1}(f_{k+1,n}(s))\Big| & \leq_c & n^{l-1} \ e^{S_n} e^{(l-1)(S_n-M_n)}. \nonumber
\end{eqnarray}
Then recalling (\ref{defH}), we have
\begin{eqnarray}
 |H^{(l)}(s)| &\leq_c& \sum_{k=0}^{n-1} n^{l-1} \ e^{-S_k} e^{(l-1)(S_n-M_n)} e^{S_n}\ \leq_c \ n^{l} e^{l(S_n-M_n)} , \nonumber
\end{eqnarray}
which gives (\ref{2306}) for $l<\lceil\beta\rceil-2$. 

We can now prove that  (\ref{2406}) is fulfilled for $l+1<\lceil\beta\rceil-1$.
Using Lemma \ref{lemder} again (see (\ref{der1})) with $f=g_{0,n}$ and $h(x)=1/x$, we get
\begin{eqnarray}
 \frac{d^l}{ds^l} g_{0,n}(s)^{-1} &=& \sum_{j=1}^{l}(-1)(-2)\cdots (-j) g_{0,n}(s)^{-(j+1)} u_{j,l}(s) \nonumber \\
&=& -g_{0,n}(s)^{-2} g_{0,n}^{(l)}(s) +\sum_{j=2}^{l}(-1)(-2)\cdots (-j) g_{0,n}(s)^{-(j+1)} u_{j,l}(s),  \label{2307}
\end{eqnarray}
where  
\begin{eqnarray}
 u_{j,l} (s) &=& \sum_{i=(i_1,\ldots,i_{2j})\in \mathcal{C}(j,l)} c_i (g_{0,n}^{(i_1)}(s))^{i_2} \cdots (g_{0,n}^{(i_{2j-1})})(s))^{i_{2j}},  \nonumber
\end{eqnarray}
and $\mathcal{C}(j,l) = \big \{ (i_1, \ldots , i_{2j})) \in \mathbb{N}^{2j} \big | i_1 i_2+ i_3 i_4 + \ldots =l \mbox{ and } \ i_2+i_4+\ldots =j\big \}$. \\
Moreover, the induction assumption (\ref{2406}) and (\ref{gf}) give for every $i\leq l-1$,
\begin{eqnarray}
 g_{0,n}^{(i)}(1)\leq_c n^{i} \ e^{S_n} e^{i(S_n-M_n)}. \nonumber
\end{eqnarray}
Thus
\begin{eqnarray}
 u_{j,l}(1) \leq_c n^{l} \ e^{jS_n} e^{l(S_n-M_n)}  .\nonumber  
\end{eqnarray}
By (\ref{defH}), the left hand-side of (\ref{2307}) is equal to $(1-s) H^{(l)}(s)-lH^{(l-1)}(s)$. By (\ref{2306}), for $l<\lceil\beta\rceil-2$, $(1-s)H^{(l)}(s)$ vanishes for $s=1$. Thus letting $s=1$  and noting that $g_{0,n}(1)=e^{S_n}$ yields
\begin{eqnarray}
 g_{0,n}^{(l)}(1) &\leq_c& e^{2S_n} \Big(\sum_{j=2}^{l}(-1)(-2)\cdots (-j) e^{-(j+1)S_n} \ n^{l} \ e^{jS_n} e^{l(S_n-M_n)} \  + l|H^{(l-1)}(1)|\Big) \nonumber \\
&\leq_c&  e^{S_n} \ n^{l} \ e^{l(S_n-M_n)} + e^{2S_n} |H^{(l-1)}(1)|. \nonumber
\end{eqnarray}
As we have already proved (\ref{2306}) for $l<\lceil\beta\rceil-2$, we get
\begin{eqnarray}
g_{0,n}^{(l)}(1) &\leq_c& \ n^{l} \ e^{S_n} e^{l(S_n-M_n)}  +  n^{l-1}\ e^{2S_n} e^{(l-1)(S_n-M_n)} \nonumber \\
&\leq_c & n^{l} e^{S_n} e^{l(S_n-M_n)} . \nonumber 
\end{eqnarray}
Using  (\ref{gf}), we get  (\ref{2406}) for $l+1$, which completes the induction and proves (\ref{2406}) for $l<\lceil\beta\rceil-1$.\\

$\qquad$ Let us  prove the bound on $H^{(l)}(s)$ for $l=\lceil\beta\rceil-2$. Using again Lemmas \ref{lem3} and \ref{lemder} and (\ref{2406}) yields
\begin{eqnarray}
&& \Big|\frac{d^l}{ds^l} h_{k+1}(f_{k+1,n}(s))\Big| \nonumber \\
&=& \Big|\sum_{j=1}^{l-1} h_{k+1}^{(j)}(f_{k+1,n}(s)) u_{j,l}(s) + h_{k+1}^{(l)}(f_{k+1,n}(s)) (f'_{k+1,n}(s))^l\Big|\nonumber \\
&\leq_c& n^{l-1} e^{S_n} e^{(\lceil\beta\rceil-2)(S_n-M_n)} + \Upsilon(1/(1-f_{k+1,n}(s))) (1-f_{k+1,n}(s))^{-(\lceil\beta\rceil-\beta)}(f'_{k+1,n}(s))^l. \label{derih}
\end{eqnarray}
Now by the same arguments as in the proof of Theorem \ref{theobond}, $\Upsilon(1/(1-f_{k+1,n}(s)))\leq \Upsilon(n^2 e^{-M_n} (1-s)^{-1})$ and by convexity,
$$(1-f_{k+1,n}(s))^{-(\lceil\beta\rceil-\beta)}\leq (1-s)^{-(\lceil\beta\rceil-\beta)}(f'_{k+1,n}(s))^{-(\lceil\beta\rceil-\beta)}.$$   
Using also $f'_{k+1,n}(s)\leq e^{S_n-M_n}$, by (\ref{derih}) follows
\begin{eqnarray}
 \Big|\frac{d^l}{ds^l} h_{k+1}(f_{k+1,n}(s))\Big| &\leq_c&  n^{l-1}e^{S_n}e^{(\lceil\beta\rceil-2)(S_n-M_n)}+\Upsilon(n^2 e^{-M_n} (1-s)^{-1})(1-s)^{-(\lceil\beta\rceil-\beta)}e^{(\beta-2)(S_n-M_n)}. \nonumber
\end{eqnarray}
Combining this inequality with (\ref{defH})  proves (\ref{1101}).\\

$\qquad$ This implies that    $(1-s)H^{(l)}(s)\rightarrow 0$ as $s\rightarrow 1$  for $l=\lceil\beta\rceil-2$.
Thus we  can apply the same arguments to get an upper bound for  $g_{0,n}^{(l)}(1)$ and
prove (\ref{2406}) for $l=\lceil\beta\rceil-1$.\\

$\qquad$ Finally, let $l=\lceil\beta\rceil-1$. We apply just the same arguments as before. Then Lemmas \ref{lem3} and \ref{lemder} yield
\begin{eqnarray}
&& \Big|\frac{d^l}{ds^l} h_{k+1}(f_{k+1,n}(s))\Big| \nonumber \\
&=& \Big|\sum_{j=1}^{l-2} h_{k+1}^{(j)}(f_{k+1,n}(s)) u_{j,l}(s) + 
 l h_{k+1}^{(l-1)}(f_{k+1,n}(s))f^{(2)}_{k+1,n}(s) (f'_{k+1,n}(s))^{l-2}  \nonumber \\
&& \qquad \qquad  \qquad \qquad+h_{k+1}^{(l)}(f_{k+1,n}(s)) (f'_{k+1,n}(s))^l\Big|\nonumber \\
&\leq_c & n^{l-1} e^{S_n} e^{(l-1)(S_n-M_n)} 
\nonumber\\ 
&& 
+n\Upsilon(n^2 e^{-M_n}(1-s)^{-1})\big(e^{(\beta-1)(S_n-M_n)} (1-s)^{-(\lceil\beta\rceil-\beta)}+e^{(\beta-2)(S_n-M_n)} (1-s)^{-1-(\lceil\beta\rceil-\beta)}\big).\nonumber
\end{eqnarray}
Using again (\ref{defH}), this proves (\ref{11012}). 
\end{proof}

\begin{proof}[Proof of Theorem \ref{theobond2} for $\beta>2$]
Let $l=\lceil\beta\rceil-1$. Without loss of generality, we assume $\Upsilon\geq 1$. The following relations hold a.s. 
Using (\ref{2307}) and (\ref{23061}),

\begin{eqnarray}
 g^{(l)}_{0,n}(s) &=& g_{0,n}(1)^2 \Big(-(1-s) H^{(l)}(s) +l H^{(l-1)}(s)+\sum_{j=2}^{l}(-1)(-2)\cdots (-j) g_{0,n}(s)^{-(j+1)} u_{j,l}(s)\Big)  \nonumber 
\end{eqnarray}
Now using (\ref{1101}),  (\ref{11012}), (\ref{2307}) as well as $\exp(S_n)\leq \exp(S_n-M_n)$ for the first terms and (\ref{2406}) together with (\ref{gf}) for the last term yields
\begin{eqnarray}
 g^{(l)}_{0,n}(s)&\leq_c& e^{S_n} n^{l} \Upsilon(n^2e^{-M_n} (1-s)^{-1})\Big( (1-s)^{-(\lceil\beta\rceil-\beta)} e^{(\beta-1)(S_n-M_n)} \nonumber \\
&&+ (1-s)^{1-(\lceil\beta\rceil-\beta)} e^{\beta(S_n-M_n)} + (1-s)e^{\lceil\beta\rceil (S_n-M_n)} \nonumber \\
&&+(1-s)^{-(\lceil\beta\rceil-\beta)} e^{(\beta-1)(S_n-M_n)} + e^{(\lceil\beta\rceil-1) (S_n-M_n)}\Big) \nonumber \\
&&+ e^{S_n} n^{l} e^{(\lfloor \beta\rfloor-1)(S_n-M_n)}. \nonumber 
\end{eqnarray}

Analogously to  (\ref{derivativeg}), we get the following estimate for every $1/2\leq s< 1$,
\begin{eqnarray}
g^{(l)}_{0,n}(s) &\geq& s^{k} \frac{k^{l+1}}{2} \P(Z_n>k|\Pi). \nonumber 
\end{eqnarray}
Choosing $s=1-1/k$ yields
\begin{eqnarray}
 \P(Z_n>k|\Pi) &\leq_c& e^{S_n} n^l \Upsilon(n^2e^{-M_n} k)\Big( k^{-\beta} e^{(\beta-1)(S_n-M_n)} 
+ k^{-(\beta+1)} e^{\beta(S_n-M_n)}\nonumber \\
 &&\qquad \qquad \qquad \qquad+ k^{-\lceil\beta\rceil-1} e^{\lceil\beta\rceil(S_n-M_n)}+k^{-\lceil\beta\rceil} e^{(\lceil\beta\rceil -1)(S_n-M_n)}\Big). \nonumber
\end{eqnarray}
Using that for all $a\geq 1$ and $b\geq 0$, the function $x\rightarrow a^{-x} \exp((x-1)b)$ is monotone 
and that $\beta\leq \lceil \beta\rceil<\beta+1\leq \lceil \beta\rceil +1$, we have for all $k\geq 1$,
$$k^{-\lceil\beta\rceil} e^{(\lceil\beta\rceil -1)(S_n-M_n)}\leq \max\big\{k^{-\beta} e^{(\beta-1)(S_n-M_n)};k^{-\lceil\beta\rceil-1} e^{\lceil\beta\rceil(S_n-M_n)}\big\}.$$
Combining the two last inequalities leads to
\begin{eqnarray}
 P(Z_n>k|\Pi) &\leq_c& e^{S_n} n^{l}  \Upsilon(n^2e^{-M_n}k) \max\big\{k^{-\beta} e^{(\beta-1)(S_n-M_n)};k^{-\lceil\beta\rceil-1} e^{\lceil\beta\rceil(S_n-M_n)}\big\}, \nonumber
\end{eqnarray}
which completes the proof.
\end{proof}

\begin{proof}[Proof of the lower bound of Theorem $1$ for $\beta>2$]
The proof now follows the proof for $\beta\in(1,2]$. Theorem 
\ref{theobond2} yields
\begin{eqnarray}
 \liminf_{n\rightarrow\infty} -\frac{1}{n} \mathbb{P}(Z_n>e^{\theta n}) &\geq& \min\big\{\psi_{\gamma, \beta, \Lambda}(\theta), \psi_{\gamma,\lceil \beta\rceil+1, \Lambda}(\theta)\big\},\nonumber
\end{eqnarray}
where $\psi$ is defined in (\ref{psidef}). Now using the characterization of $\psi$ given in forthcoming Lemma \ref{exprPsi}, we deduce that for any $\theta\geq 0$,
$$\psi_{\gamma,\beta,\Lambda}(\theta) \leq \psi_{\gamma,\lceil \beta\rceil+1,\Lambda}(\theta).$$
Thus 
$$\min\big\{\psi_{\gamma,\beta,\Lambda}(\theta), \psi_{\gamma,\lceil \beta\rceil+1,\Lambda}(\theta)\big\}=\psi_{\gamma,\beta,\Lambda}(\theta) =\psi(\theta)$$
and we get the expected lower bound.
\end{proof}

\section{Proof of the corollary}
By assumption,  there exists a constant $d<\infty$ such that for every $\beta>0$,
\begin{eqnarray}
 \mathbb{P}(L>z|f,L>0) &\leq& d \cdot (m\wedge 1)\cdot z^{-\beta} \quad \mbox{a.s.}\nonumber 
\end{eqnarray}
Then we can apply the lower bound in Theorem  \ref{theo1} for every $\beta>0$. This yields for all $\beta>0$ and $\theta\geq 0$,
\begin{eqnarray}
 \liminf_{n\rightarrow\infty} -\frac{1}{n} \log \mathbb{P}(Z_n>e^{\theta n}) &\geq& \psi_{\gamma, \beta,\Lambda}(\theta). \nonumber 
\end{eqnarray}
Now taking the limit $\beta\rightarrow\infty$, the monotone convergence of $\psi_{\gamma,\beta, \Lambda}$ 
 yields
\begin{eqnarray}
 \liminf_{n\rightarrow\infty} -\frac{1}{n} \log \mathbb{P}(Z_n>e^{\theta n}) &\geq& \psi_{\gamma, \infty,\Lambda}(\theta), \nonumber 
\end{eqnarray}
where 
\Bea
\psi_{\gamma,\infty,\Lambda}(\theta)&:=& \lim_{\beta\rightarrow\infty} \inf_{t\in[0,1],  s\in [0,\theta]} \Big\{t\gamma+\beta s+(1-t)\Lambda((\theta-s)/(1-t))\Big\}\\
&=&\inf_{t\in[0,1]} \Big\{t\gamma+(1-t)\Lambda(\theta/(1-t))\Big\}. \Eea
This gives the upper bound and the lower bound follows readily the proof given in Section $3$ where we consider the natural associated path (or see \cite{BK09}). \qed

\section{Appendix}
We give in this section several technical results useful for the proofs.
\subsection{Characterization of the rate function $\psi$}
\begin{lemma} \label{exprPsi}
Let $0\leq \gamma\leq \Lambda(0)$ and $\beta> 0$. The  function $\psi$ defined
for $\theta\geq 0$ by 
$$\psi(\theta)=\inf_{t\in[0,1], \\  s\in [0,\theta]}\big \{t\gamma+\beta s+(1-t)\Lambda((\theta-s)/(1-t))\big\}$$
  is the largest convex function such that for all $x,\theta \geq 0$
\begin{eqnarray}
 \psi(0)=\gamma,  \quad \psi(\theta)\leq \Lambda(\theta),  \quad \psi(\theta+x)\leq \psi(\theta)+\beta x. \label{propPsi}
\end{eqnarray}
\label{lempsi}
\end{lemma}
\begin{proof} First, we  prove that $\psi$ is convex.
 Using the definition of $\psi$ and the convexity of $\Lambda$,  for any $\theta', \theta'' \ge 0$ and $\epsilon > 0$ there exist $t', t'' \in [0,1)$ and $s'\in [0,\theta']$, $s''\in [0,\theta'']$,  such that 
for every $\lambda\in [0,1]$,
\Bea
&&\lambda  \psi(\theta') +  (1-\lambda ) \psi(\theta'') \\
&& \quad \ge  \lambda  [t'\gamma +  \beta s'+ (1-t') \Lambda ((\theta'-s')/(1-t'))] \\
&& \qquad \qquad \qquad + (1-\lambda )[
 t''\gamma +  \beta s''+ (1-t'') \Lambda ((\theta''-s'')/(1-t''))] - \epsilon  \\
&&\quad \ge [\lambda t'+(1-\lambda )t''] \gamma + [\lambda s'+(1-\lambda )s''] \beta \\
&&\qquad \qquad \qquad  +\Big(1-[\lambda t'+(1-\lambda )t'']\Big)  \Lambda \Big(\frac{\lambda\theta'+(1-\lambda )\theta''-(\lambda s'+(1-\lambda )s'')}{1-[\lambda t'+(1-\lambda )t'']}\Big) - \epsilon\\
&& \quad \ge \psi\Big( \lambda \theta'+ (1-\lambda )\theta''\Big) - \epsilon.
\Eea
Letting $\epsilon\rightarrow 0$ entails that $\psi$ is convex. \\

Second, following the previous computation, we verify that $\psi$ fulfills (\ref{propPsi}). For any $\theta\geq 0$ and $\epsilon>0$, there exist $t'\in[0,1)$ and $s'\in[0,\theta]$ such that
\begin{eqnarray}
 \psi(\theta) & \geq & t'\gamma + \beta s'+(1-t')\Lambda\big((\theta-s')/(1-t')\big)-\epsilon \nonumber \\
&=&  t'\gamma + \beta  (s'+x)+(1-t')\Lambda\big((\theta+x- (s'+x))/(1-t')\big)-\beta x-\epsilon \nonumber \\
&\geq& \inf_{t\in[0,1],\tilde s\in [0,\theta+x]} \big\{t\gamma+\beta \tilde{s}+(1-t)\Lambda((\theta+x-\tilde s)/(1-t))\big\}-\beta x-\epsilon.\nonumber
\end{eqnarray}
Taking the limit $\epsilon\rightarrow 0$ yields the second property in (\ref{propPsi}). Furthermore,
letting $t=0$ and $s=0$ implies $\psi(\theta)\leq \Lambda(\theta)$ and  $t\rightarrow 1$ entails that $\psi(0)\leq \gamma$. This completes the proof of  (\ref{propPsi}).\\

$\qquad$ Finally, let $\kappa$ be any convex function which satisfies (\ref{propPsi}). Using these assumptions ensures that for all  $t\in[0,1)$ and $0\leq s\leq \theta$, 
\begin{eqnarray*}
 t\gamma + \beta s + (1-t) \Lambda((\theta-s)/(1-t)) & \ge &t\kappa(0) + \beta s+(1-t) \kappa((\theta-s)/(1-t)) \\
& \ge & \beta s+ \kappa\big( t0 + (1-t) (\theta-s)/(1-t)) \big) \\
&=& \beta s + \kappa(\theta-s) \ \nonumber \\
&\geq& \kappa(\theta). \nonumber
\end{eqnarray*}
Taking the infimum over $s$ and $t$, we get $\psi(\theta)\geq \kappa(\theta)$ and the proof is complete. 
\end{proof}
We give now describe a last characterization of $\psi$ that results from Lemma \ref{lempsi} (see Figure \ref{ratefig}). Let $\theta^*$ and $\theta^\dagger$ be defined as in (\ref{thetastar}) and (\ref{thetadagger}) and assume $0<\theta^*<\theta^{\dagger}<\infty$. As convex and monotone function, $\Lambda$ has at most one jump (to infinity). Let this jump be in $0<\theta_j\leq \infty$ and $\Lambda(\theta)$ is differentiable for $\theta<\theta_j$. As $\Lambda$ is also continuous from below, $\Lambda(\theta_j)<\infty$. Now, by the preceding characterization, $\psi$ is the largest convex function, starting in $\psi(0)=\gamma$, being at most as large as $\Lambda$ and having at most slope $\beta$.\\
The largest convex function through the point $(0,\gamma)$ being smaller/equal than $\Lambda$ has to be linear and has to be a tangent of $\Lambda$. By definition of $\theta^*$, the tangent at $\Lambda$ in $\theta^*$ goes through the point $(0,\gamma)$. Thus $\psi$ is linear for $\theta<\theta^*$ and follows this tangent. For $\theta>\theta^*$, $\psi$ is identical with $\Lambda$ until the slope of $\Lambda$ is exactly $\beta$ (or until $\Lambda$ jumps to infinity). At this point $\theta^\dagger$, the last condition becomes important and $\psi$ is linear with slope $\beta$ for $\theta>\theta^\dagger$. Summing up,
\begin{eqnarray}
 \psi(\theta) &=& \left\{ \begin{array}{l@{\quad,\quad}l}
                             \gamma(1-\frac{\theta}{\theta^*}) + \frac{\theta}{\theta^*} \Lambda(\theta^*) & \mbox{if} \ \theta\leq \theta^{*} \\
\Lambda(\theta) & \mbox{if} \ \theta^*<\theta< \theta^{\dagger}\\
 \beta(\theta-\theta^\dagger) + \Lambda(\theta^\dagger) & \mbox{if} \ \theta\geq \theta^{\dagger} 
                             \end{array} \right. \ . \nonumber
\end{eqnarray}
If $\Lambda'(0)>\beta$, then $\theta^\dagger=0$ and $\psi(\theta)=\gamma + \beta \theta$. If $\gamma=\Lambda(0)$, then $\theta^*=0$. We refrain from describing other degenerated cases. 

\subsection{Slowly varying functions}
In this section, we recall some properties of regularly varying functions
and we refer to \cite{BGT} for details. The function
$\Upsilon: (0,\infty)\rightarrow (0,\infty)$ is a  slowly varying function if for every $a>0$, $$\lim_{x\rightarrow\infty} \frac{\Upsilon(a  x)}{\Upsilon(x)}=1.$$ 

We need a Tauberian result   from \cite{feller2}, p. 423. 
See also \cite{BGT}, Theorem 1.5.11, page 28.
 For any $\alpha>-1$, the function 
 $g(s):=\sum_{k=0}^{\infty} s^k k^{\alpha}$ 
satisfies
$$
 g(s)\sim \Gamma (\alpha+1) (1-s)^{-1-\alpha}  \quad (s\rightarrow 1-).
$$
Then the function 
  $\xi = s\rightarrow (1-s)^{1+\alpha}g(s)$ is continuous on $[0,1)$ and has a finite limit in $1-$. Denoting by $M$ the  supremum of this function extended to $[0,1]$, we get
$$  \sum_{k=1}^{\infty} s^k k^{\alpha}\leq M  (1-s)^{-1-\alpha} \qquad (0\leq s<1). $$
For $\alpha=-1$,  $\sum_{k=1}^{\infty} s^k/k =-\log (1-s)$. As  the logarithm is a slowly varying function, we rewrite the previous results in the following way, which will be convenient in the proofs. \\
There exists a nondecreasing positive slowing varying function $\Upsilon$  such that for all
$\alpha\geq -1$ and $s\in [0,1)$

\begin{eqnarray}\label{bound1}
  \sum_{k=1}^{\infty} s^k k^{\alpha}&\leq&  \Upsilon(1/(1-s)) (1-s)^{-1-\alpha}. 
\end{eqnarray}

\subsection{Bounds for generating functions}
Let $L$ be a random variable with values in $\{0,1,2,...\}$ with expectation $m$, distribution $(p_k)_{k\in\mathbb{N}}$ and generating function $f$. Let us define 
$$q_k:=\mathbb{P}(L>k|f)$$
and the  following function associated to $f$,
\be
\label{defg}
g(s):=\sum_{k=0}^{\infty} s^k q_k =\frac{1-f(s)}{1-s},
\ee
where   the last identity comes from Cauchy product of power series (see also \cite{BK09}).
We recall that  the $l$-the derivative of a function $f$ is denoted by $f^{(l)}$ and that
$f^{(l)}(s)$ and $g^{(l)}(s)$ exist  for every $s\in[0,1)$.  As
$$
 f^{(l)}(s) = \sum_{k=0}^{\infty} k(k-1) \cdots (k-l+1) s^{k-l} p_k, \qquad g^{(l)}(s) = \sum_{k=0}^{\infty} k(k-1) \cdots (k-l+1) s^{k-l} q_k, 
 $$
all derivatives of $f$ and $g$ are nonnegative, nondecreasing functions. We are using $g$ instead of $f$ in the proofs since 
 the associated sequence $(q_k)_{k\in\mathbb{N}}$ is monotone, which is more convenient. Calculating the $l$-th derivative of
$f(s) = 1-(1-s)g(s)$
gives
\begin{eqnarray}
  f^{(l)}(s) &=& l g^{(l-1)}(s)-(1-s) g^{(l)}(s) . \label{deg}
\end{eqnarray}
Thus $g^{(l-1)}(1)$ and $f^{(l)}(1)$ both essentially describe the $l$-th moment of the corresponding probability distribution. More precisely, if $g^{(l-1)}(1)$ is finite, then $f^{(l)}(1)$ is finite and
\begin{eqnarray}
 f^{(l)}(1)  &= & l g^{(l-1)}(1) . \label{gf} 
\end{eqnarray}
Conversely if  $f^{(l)}(1)<\infty$, then 
$$g^{(l-1)}(1)= \sum_{k=1}^{\infty} k(k-1)\cdots(k-l+1) q_k \leq \sum_{k=1}^{\infty} k^l p_k<\infty$$ and $g^{(l-1)}(1)<\infty$. \\

 For $\mu\in(0,1]$, we also define the function 
\be 
h_{\mu}(s):= \frac{g(1)^{\mu}-g(s)^{\mu}}{(g(1)g(s)(1-s))^{\mu}}. \label{1301}\\
\ee
The following useful lemmas   give  versions of assumption $\mathcal{H}(\beta)$ in terms of the function $h_{\mu}$.
  Noting that $g(0)=q_0=\mathbb{P}(L>0|f)$ and $g(1)=m$, we can rewrite
assumption $\mathcal{H}(\beta)$ in the following way
\be
\label{versionH}
q_k \leq d \ g(0) \ (g(1)\wedge1) \ k^{-\beta} \qquad (k\geq 1).  	 
\ee
\begin{lemma}\label{bound2} Let  $\beta>1$ and assume that (\ref{versionH}) holds for some  constant  $0<d<\infty$.
Then  for every  $0<\mu<(\beta-1)\wedge 1$, there exists a constant $c=c(\beta,d,\mu)$ such that for every $s\in[0,1]$,
\begin{eqnarray}
 h_{\mu} (s) &\leq& c. \label{eqth}
\end{eqnarray}
The above bound also holds for $\mu=1$ if $\beta>2$. 
Moreover, if $\beta\in(1,2]$, there exists a nondecreasing positive slowly varying function $\Upsilon=\Upsilon(\beta,d)$  such that, for every $s\in[0,1)$, 
\begin{eqnarray}
 h_{\beta-1}(s) &\leq&  \Upsilon(1/(1-s)) \label{eqhmu} \\
 -h_{\beta-1}'(s) &\leq&  \Upsilon(1/(1-s)) /(1-s). \label{eqhmup} 
\end{eqnarray}
\end{lemma}
Note that $\Upsilon$ depends on $L$ (or $g$) only through the values of $d$ and  $\beta$. Then under assumption  $\mathcal{H}(\beta)$, we derive from this lemma  a nonrandom  constant bound. \\

$\quad$ In the proofs, we use again the notation $\leq_c$ which means that the inequality is fulfilled up to a multiplicative constant which depends on $\beta$ and $\mu$ but is independent of $s$ and the order of the differentiation.
\begin{proof}
 Using $g(s)\geq g(0)$, we have
\begin{eqnarray}
h_{\mu}(s) & = & \frac{g(1)^{\mu}-g(s)^{\mu}}{(g(1) g(s) (1-s))^{\mu}} \nonumber \\
& \leq & \frac{g(1)^{\mu}-g(s)^{\mu}}{(g(1) g(0) (1-s))^{\mu}}  \nonumber \\
  &\leq & (g(1)\wedge  1)^{-1}\frac{(\sum_{k=0}^{\infty} g(0)^{-1}  q_k)^{\mu}-(\sum_{k=0}^{\infty}s^k q_k g(0)^{-1})^{\mu}}{(1-s)^{\mu}}.\label{new1}
\end{eqnarray}

  Since $\mu\in (0,1]$, the function $x\rightarrow x^{\mu}$ is concave, so that
$ a^\mu-x^\mu \leq \mu x^{\mu-1} (a-x)$
for all $0\leq x\leq a$.  Moreover  
\begin{eqnarray}
1=q_0/g(0)\leq  x:= \sum_{k=0}^{\infty}s^k q_k g(0)^{-1} \leq a:=\sum_{k=0}^{\infty}  q_k g(0)^{-1}. \label{est1}
\end{eqnarray}
Then $x^{\mu-1}\leq 1$ and using  the inequality of concavity  in $(\ref{new1})$ with  $q_k \leq d g(0) \cdot (g(1)\wedge1)\cdot k^{-\beta}$ leads to
\Bea
h_{\mu}(s)&\leq &\mu (g(1)\wedge  1)^{-1}x^{\mu-1}\frac{\sum_{k=0}^{\infty} g(0)^{-1}q_k [1-s^k]}{(1-s)^{\mu}} \\
&\leq_c& \frac{\sum_{k=1}^{\infty} (1-s^k) k^{-\beta}}{(1-s)^{\mu}}\nonumber \\
&=& (1-s)^{1-\mu} \sum_{k=1}^{\infty} \frac{1-s^k}{1-s} k^{-\beta} \nonumber \\
&=& (1-s)^{1-\mu} \sum_{k=1}^{\infty}  k^{-\beta} \sum_{j=0}^{k-1} s^j \nonumber \\
&=& (1-s)^{1-\mu} \sum_{j=0}^{\infty}  s^j \sum_{k=j+1}^{\infty} k^{-\beta} \nonumber \\
&\leq_c& (1-s)^{1-\mu} \sum_{j=0}^{\infty}  s^j (j+1)^{-\beta+1}. \nonumber 
\Eea
The estimates (\ref{eqth}) and (\ref{eqhmu}) on $h_{\mu}$ for $0<\mu<(\beta-1)\wedge 1$ and $\mu=\beta-1$ now follow directly from  (\ref{bound1}). For $\mu=1$, $\beta>2$ and $s=1$, the sum is finite and (\ref{eqth}) also holds in this case.\\

$\quad$ For the second part of the lemma, we explicitly calculate the first derivative of $h_{\beta-1}$ by using the formula 
\begin{eqnarray}
 h_{\beta-1}(s) g(s)^{\beta-1}  &=& \frac{g(1)^{\beta-1}-g(s)^{\beta-1}}{g(1)^{\beta-1}(1-s)^{\beta-1}}. \nonumber
\end{eqnarray}
Differentiating both sides yields
$$
h_{\beta-1}'(s) g(s)^{\beta-1}+(\beta-1)h_{\beta-1}(s) g(s)^{\beta-2} g'(s)= \frac{(\beta-1) ([g(1)^{\beta-1}-g(s)^{\beta-1}]-(1-s)g(s)^{\beta-2} g'(s))}{g(1)^{\beta-1}(1-s)^{\beta}}
$$
and thus
$$
 -h_{\beta-1}'(s) \leq (\beta-1)\Big(\frac{ h_{\beta-1}(s) g'(s)}{g(s)} + \frac{g'(s)}{g(s)g(1)^{\beta-1}(1-s)^{\beta-1}}-\frac{g(1)^{\beta-1}-g(s)^{\beta-1}}{g(s)^{\beta-1}g(1)^{\beta-1}(1-s)^{\beta}}\Big)$$

As $g$ is nondecreasing, we can skip the the last term which is negative. Using (\ref{versionH}) and (\ref{eqhmu}), we get
$$
 -h_{\beta-1}'(s) \leq_c 
 \frac{g(0)\cdot(g(1)\wedge 1)\cdot \sum_{k=1}^{\infty} ks^{k-1} k^{-\beta}}{g(s)}\Big( \Upsilon(1/(1-s))
+ \frac{1}{g(1)^{\beta-1}(1-s)^{\beta-1}} \Big).
$$
Moreover  $g(s)\geq g(0)$ and   $g(1)^{-(\beta-1)}\cdot(g(1)\wedge 1)\leq 1$ for $\beta-1\in(0,1]$, so
$$-h_{\beta-1}'(s)
\leq_c \sum_{k=1}^{\infty} s^{k-1} k^{-\beta+1}\Big(\Upsilon(1/(1-s))
+ \frac{1}{(1-s)^{\beta-1}}\Big).
$$
  The result now follows from (\ref{bound1}) and the fact that the product of two slowly varying functions is still slowly varying.
\end{proof}

We consider now
$$h(s)=h_1(s)=\frac{g(1)-g(s)}{g(1)g(s)(1-s)}.$$

\begin{lemma}\label{lem3} We assume that (\ref{versionH})
holds for some $\beta>1$.
Then there exists a finite constant $c=c(\beta,d)<\infty$ such that for every $s\in[0,1)$,
\begin{eqnarray}
 |h^{(l)}(s)| &\leq &  c \qquad \quad  \quad \quad \quad \quad \quad \quad \quad \quad \quad \quad \ \mbox{if} \quad 0\leq l<\beta-2 \nonumber \\
|h^{(\lceil \beta\rceil-2)} (s)| &\leq & c  \Upsilon(1/(1-s)) \ (1-s)^{-(\lceil\beta\rceil-\beta)} \quad \mbox{if} \quad \beta\geq 2  \nonumber \\
|h^{(\lceil \beta\rceil-1)} (s)| &\leq & c  \ \Upsilon(1/(1-s)) \ (1-s)^{-1-(\lceil\beta\rceil-\beta)}. \label{as11}
\end{eqnarray}
\end{lemma}
\begin{proof} 
By (\ref{1301}) and Cauchy product of power series, for every $s\in [0,1)$,
\begin{eqnarray}
g(s) g(1) h(s) &=& \frac{g(1)-g(s)}{1-s} \ = \ \sum_{k=0}^{\infty} s^k (q_{k+1}+q_{k+2}+\ldots).  \nonumber 
\end{eqnarray}
Thus, the $l$-th derivative of $g(s)h(s)$ is 
\begin{eqnarray}
 \sum_{j=0}^{l} {l\choose j} g^{(j)}(s) h^{(l-j)}(s) &=& g(1)^{-1} \sum_{k=0}^{\infty} k(k-1)\cdots (k-l+1)s^{k-l} (q_{k+1}+q_{k+2}+\ldots). \nonumber  
\end{eqnarray}
Moreover,  (\ref{versionH}) ensures that for all $s\in[0,1)$ and $j<\beta-2$,
\begin{eqnarray}
 g^{(j)}(s)\leq  g^{(j)}(1) \ \leq \ \sum_{k=0}^{\infty} k^{j} q_k \ \leq_c g(0) (g(1)\wedge 1) \nonumber  
\end{eqnarray}
Combining the two last expressions
and using $g(s)^{-1}\leq g(0)^{-1}$ gives
\bea
|h^{(l)}(s)|&\leq_c& g(s)^{-1}\Big(g(1)^{-1} g(0) \cdot (g(1)\wedge 1)\cdot  \sum_{k=0}^{\infty} k^l s^{k-l} \sum_{j=k+1}^{\infty} j^{-\beta}+ \sum_{j=1}^{l}  {l\choose j} g^{(j)}(1) |h^{(l-j)}(s)|\Big) \nonumber \\
&\leq_c&\sum_{k=0}^{\infty} k^l s^{k-l} \sum_{j=k+1}^{\infty} j^{-\beta}+ \sum_{j=1}^{l}  {l\choose j} |h^{(l-j)}(s)|\Big) \label{p1} 
\eea
We can  prove the first statement of the lemma by induction on $l$. For $l=0$, it is  given by Lemma \ref{bound2}.  Assuming that the bounds holds for $l'< l<\beta-2$, the previous inequality ensures that
$$
|h^{(l)}(s)|
\leq_c 1 \  + \sum_{j=0}^{l-1} |h^{(j)}(s)|$$
since  $\sum_{k=0}^{\infty} k^l\sum_{j=k+1}^{\infty} j^{-\beta}<\infty$. This ends up the induction and proves the first estimate in (\ref{as11}). \\

We consider now  $l=\lceil \beta \rceil-2$ and 'continue the induction'. Using the bound of $h^{(l)}$ for $l<\beta-2$ and (\ref{p1})  yields
$$|h^{(l)}(s)|\leq_c\sum_{k=0}^{\infty} k^l s^k \sum_{j=k+1}^{\infty} j^{-\beta}  + 1
\leq_c \sum_{k=1}^{\infty} s^k k^{\lceil \beta \rceil-2} k^{-\beta+1}  + 1 
\leq_c \sum_{k=1}^{\infty} s^k k^{-(1-(\lceil\beta\rceil-\beta))}.$$
Then the second estimate of the lemma  follows from (\ref{bound1}). \\

Finally, we prove the bound for  $l=\lceil \beta \rceil -1$ in the same way. By (\ref{p1}):
\begin{eqnarray}
 |h^{(l)}(s)|&\leq_c&\sum_{k=0}^{\infty} k^l s^k \sum_{j=k+1}^{\infty} j^{-\beta} +  l g^{(2)}(1) |h^{(l-1)}(s)| + 1 \nonumber \\
&\leq_c&\sum_{k=1}^{\infty} s^k k^{\lceil \beta \rceil-\beta} +\sum_{k=1}^{\infty} s^k k^{-(1-(\lceil\beta\rceil-\beta))} + 1\nonumber \\
&\leq_c&\sum_{k=1}^{\infty} s^k k^{\lceil\beta\rceil-\beta} \ .\nonumber 
\end{eqnarray}
and Lemma \ref{bound1} allows us to conclude. 
\end{proof}

\subsection{Successive Differentiation for composition of functions}
For the proof of the upper bound on the tail probabilities when $\beta>2$, we need to calculate higher order derivatives of a composition of functions. Here we prove a useful formula for the $l$-th derivative of a composition of two functions, which could also be derived from the combinatorial form of Faà di Bruno's formula. 
\begin{lemma}
Let $f$ and $h$ be real-valued, $l$-times differentiable functions. Then 
\begin{eqnarray}
 \frac{d^l}{ds^l} h(f(s)) &=& \sum_{j=1}^{l} h^{(j)}(f(s)) u_{j,l}(s), \label{der1}
\end{eqnarray}
where $u_{j,l}(s)$ is given by
\begin{eqnarray}
 u_{j,l} (s) &=& \sum_{i=(i_1,\ldots,i_{2j})\in \mathcal{C}(j,l)} c_i (f^{(i_1)}(s))^{i_2} \cdots (f^{(i_{2j-1})})(s))^{i_{2j}},  \label{der2} 
\end{eqnarray}
with some constants $0\leq c_i<\infty$  and  $\mathcal{C}(j,l)$  defined by
\begin{eqnarray}
 \mathcal{C}(j,l) &:=& \big \{ (i_1, \ldots , i_{2j}) \in \mathbb{N}^{2j} \big | i_1 i_2+ i_3 i_4 + \ldots =l \mbox{ and } \ i_2+i_4+\ldots =j\big \}. \nonumber
\end{eqnarray}
 \label{lemder}
\end{lemma}
\begin{proof} We prove the formula by induction with respect to $l$. For $l=1$, by chain rule of differentiation, (\ref{der1}) is fulfilled. Assume that (\ref{der1}) and (\ref{der2}) hold for $l$. Then by product rule for differentiation,
\begin{eqnarray}
  \frac{d^{l+1}}{ds^{l+1}} h(f(s)) &=& \sum_{j=1}^l \Big(h^{(j)}(f(s)) \frac{d}{ds} u_{j,l}(s) + u_{j,l}(s) f'(s) h^{(j+1)}(f(s))\Big) \ . \nonumber 
\end{eqnarray}
Now 
\begin{eqnarray}
u_{j,l}(s) f'(s) &=& \sum_{i\in\mathcal{C}(j,l)} c_i \big(f^{(1)}(s)\big)^1 (f^{(i_1)}(s))^{i_2} \cdots (f^{(i_{2j-1})}(s))^{i_{2j}} \nonumber \\
&=&\sum_{i\in\mathcal{C}(j+1,l+1)} \tilde c_i (f^{(i_1)}(s))^{i_2} \cdots (f^{(i_{2j+1})}(s))^{i_{2(j+1)}} \ , \nonumber 
\end{eqnarray} 
with new constants given by
\begin{eqnarray}
\tilde c_{i_1,i_2,i_3,\ldots, i_{2(j+1)}} &:=& \left\{ \begin{array}{l@{\quad,\quad}l}
                              c_{i_3,\ldots,i_{2(j+1)}} & \mbox{if} \ i_1=i_2=1 \\
0 & \mbox{else}
                             \end{array} \right. \ . \nonumber 
\end{eqnarray}
Furthermore,
\begin{eqnarray}
 \frac{d}{ds} u_{j,l}(s) 
 &=& \sum_{i\in\mathcal{C}(j,l)} \sum_{k=1}^{l} c_i (f^{(i_1)}(s))^{i_2} \cdots \ i_{2k}(f^{(i_{2k-1})}(s))^{i_{2k}-1} f^{(i_{2k-1}+1)}(s) \cdots (f^{(i_{2j-1})}(s))^{i_{2j}} \nonumber \\
&=& \sum_{i\in\mathcal{C}(j+1,l+1)} \hat c_i (f^{(i_1)}(s))^{i_2} \cdots (f^{(i_{2j+1})}(s))^{i_{2(j+1)}},\nonumber
\end{eqnarray}
with some new constants $0\leq \hat c_i <\infty$. This ends up the induction. \end{proof}

\bibliographystyle{apalike}
\end{document}